\documentclass[11pt]{article}
\usepackage{amsfonts,amsmath,amssymb}
\usepackage{url}
\usepackage{appendix}
\usepackage{amsthm}
\usepackage{epsfig,latexsym,graphicx}
\usepackage{esint}
\usepackage{graphicx,latexsym,amssymb,amsmath,amsfonts,fancyhdr}
\usepackage{euscript}
\newtheorem{maintheorem}{Theorem}

\newtheorem{theorem}{Theorem}[section]
\newtheorem{lemma}[theorem]{Lemma}

\newtheorem{proposition}[theorem]{Proposition}
\newtheorem{observation}[theorem]{Observation}

\newtheorem{definition}[theorem]{Definition}

\def\XXint#1#2#3{{\setbox0=\hbox{$#1{#2#3}{\int}$ }
\vcenter{\hbox{$#2#3$ }}\kern-.6\wd0}}

\newcommand{\E}{\mathbb{E}}
\renewcommand{\P}{\mathbb{P}}

\newcommand{\Z}{\mathbb{Z}}

\newcommand{\R}{\mathbb{R}}
\newcommand{\cf}{\mathcal{F}}

\begin{document}
\title{Bi-Lipschitz Expansion of Measurable Sets}

\author{Riddhipratim Basu\thanks{Dept. of Statistics, University of California, Berkeley. Supported by UC Berkeley Graduate Fellowship. Email:riddhipratim@stat.berkeley.edu}
\and
Vladas Sidoravicius\thanks{IMPA, Estrada Dona Castorina 110, Rio de Janeiro, Brasil. Email: vladas@impa.br}
\and
Allan Sly\thanks{University of California, Berkeley and Australian National University. Supported by an Alfred Sloan Fellowship and NSF grants DMS-1208338, DMS-1352013. Email:sly@stat.berkeley.edu}
}

\date{}
\maketitle

\begin{abstract}
We show that for $0<\gamma, \gamma' <1$ and for measurable subsets of the unit square with Lebesgue measure $\gamma$ there exist bi-Lipschitz maps with bounded Lipschitz constant (uniformly over all such sets) which are identity on the boundary and increases the Lebesgue measure of the set to at least $1-\gamma'$.
\end{abstract}

\section{Introduction}
This  is the first of a series of three papers considering the problem of finding Lipschitz embeddings between high dimensional random objects. In the present paper the main result concerns proving the existence of localized bi-Lipschitz stretchings of arbitrary measurable sets in the unit square.

Ultimately in an upcoming work~\cite{BSS14} we prove that two Poisson process in the plane are almost surely roughly isometric, a notion of embeddings between metric spaces (see Definition~\ref{d:roughIsometry}).  In \cite{BS14} we showed the analogue of this result for one dimensional Poisson processes.  For such an embedding one must be able to match the gaps of points on all scales.  Even in one dimension this is challenging and required an involved multi-scale proof.  Our program is to construct a similar multi-scale induction in higher dimensions.  However, unlike the one dimensional case, where regions without points are simply line segments, on larger scales these sets can resemble any measurable set.  The result presented here, a purely deterministic one, provides the necessary input in establishing the base case estimates of our multi-scale induction~\cite{BSS14}. It allows us to effectively shrink the size of empty regions, see Theorem~\ref{t:riuse} below.

Our main result in this paper concerns bijections from the unit square on the plane to itself that are identity on the boundary. For any set $A$ of a fixed Lebesgue measure $\gamma$,  we are interested in constructing such a function, which additionally is bi-Lipschitz, and  maps $A$ to  a set  with Lebesgue measure above a fixed threshold $1-\gamma'$. For a fixed choice of $\gamma$ and $\gamma'$ (with $1-\gamma'>\gamma$) we want to maintain a uniform control over the bi-Lipschitz constants of the such functions. The main result in this paper, Theorem~\ref{t:stretching}, establishes that it is possible to do so.

\subsection{Statement of the Results}
We denote the sigma algebra of all Borel-measurable subsets of $[0,1]^2$ by $\mathcal{B}([0,1]^2)$ and let $\lambda$ denote the Lebesgue measure on $\R^2$. Our main result in this paper is the following.

\begin{maintheorem}
\label{t:stretching}
For each $\gamma, \gamma' \in (0,1)$, $\gamma+\gamma'<1$, there exists $C_0=C_0(\gamma, \gamma')>0$ such that for all $A\in \mathcal{B}([0,1]^2)$ with $\lambda(A)=\gamma$, there exists a bijection $\phi_0: [0,1]^2\rightarrow [0,1]^2$ such that
\begin{enumerate}
\item
$\phi_0$ is $C_0$-bi-Lipschitz, i.e.
$$\frac{1}{C_0}|x-y|\leq |\phi_0(x)-\phi_0(y)|\leq C_0|x-y|~~\forall x,y\in [0,1]^2.$$
\item
$\phi_0$ is identity on the boundary, i.e., $\phi_0(x)=x$, for all $x\in \partial [0,1]^2$.
\item
$\lambda(\phi_0(A))\geq 1-\gamma'$.
\end{enumerate}
\end{maintheorem}

As mentioned above Theroem \ref{t:stretching} plays a crucial role in the study of rough isometries of i.i.d. copies of 2-dimensional Poisson point process. Informally, two metric spaces are roughly isometric if their metrics are equivalent up to multiplicative and additive constants, a relaxation of a bi-Lipschitz mapping.  The formal definition, introduced by Gromov~\cite{Gromov:81} in the case of groups and more generally by Kanai~\cite{Kanai:85}, is as follows.

\begin{definition}\label{d:roughIsometry}
We say two metric spaces $X$ and $Y$ are roughly isometric (or, quasi isometric) with parameters $(M,D,C)$ if there exists a mapping $T:X\to Y$ such that for any $x_1, x_2 \in X$,
\[
\frac1M d_X(x_1,x_2)-D \leq d_Y(T(x_1),T(x_2)) \leq M d_X(x_1,x_2) + D,
\]
and for all $y\in Y$ there exists $x\in X$ such that $d_Y(T(x),y)\leq C$.
\end{definition}

The importance of Theorem \ref{t:stretching} in proving rough isometry of Poisson processes is illustrated by our next result.

Let $X$ and $Y$ be two independent Poisson point processes on $\R^2$. Let $X_n$ (resp. $Y_n$) denote the random metric space formed by points of $X$ (resp. $Y$) within $[0,n]^2$ along with the boundary of $[0,n]^2$. Let $X_n \hookrightarrow_{(M,D,C)} Y_n$ denote the event that there exists a rough isometry with parameters $(M,D,C)$ between $X_n$ and $Y_n$ which is identity on the boundary of $[0,n]^2$. Let $k_X(n)$ denote the number of unit squares in $[0,n]^2$ which contains at least one point of $X$. We have the following theorem as a consequence of Theorem \ref{t:stretching}.

\begin{maintheorem}
\label{t:riuse}
Fix $\epsilon, \delta>0$. Then there exist positive constants $M,D,C$ depending on $\epsilon$ and $\delta$ (not depending on $n$) such that we have that for all $n$ sufficiently large and for all $X$ with $k_X(n)\geq \delta n^2$, $\P[X_n\hookrightarrow _{(M,D,C)} Y_n \mid X_n]\geq e^{-\epsilon n^2}$.
\end{maintheorem}

\bigskip

By way of proving Theorem \ref{t:stretching} we also establish the following result which shows that it is possible to increase the measure of a set by an arbitrarily small amount in a bi-Lipschitz manner.
\begin{maintheorem}
\label{t:stretchingfixed}
Fix $0<\gamma < 1 -\gamma' <1$, and $\eta>0$. Then there exists  $\varepsilon=\varepsilon(\gamma, \gamma', \eta)>0$ such that the following holds. For every Borel set $A\subseteq [0,1]^2$ with $\lambda(A)\in [\gamma, 1-\gamma']$, there exists a bijection  $\phi=\phi_{A}: [0,1]^2\rightarrow [0,1]^2$ such that
\begin{enumerate}
\item
$\phi$ is $(1+\eta)$-bi-Lipschitz.
\item
$\phi$ is identity on the boundary of $[0,1]^2$.
\item
$\lambda(\phi(A))\geq \lambda(A)+\varepsilon.$
\end{enumerate}
\end{maintheorem}

\subsection{Related Works}
The question of existence of bi-Lipschitz homeomorphisms between different subsets of $\R^n$, satisfying certain conditions, has been studied classically, notably by McMullen \cite{CM97}. It is shown in that paper that there exists a positive function $f$ on $\R^n$, $n\geq 2$, such that there does not exist a bi-Lipschitz homeomorphism $\phi:\R^n\rightarrow \R^n$ with ${\rm{det}}~D\phi=f$. It is also proved that there exists separated nets in $\R^n$ ($n\geq 2$) which cannot be mapped into $\Z^n$ in a bi-Lipschitz way.

Shortly before uploading this paper we came across a very recent paper \cite{Kneuss} on arXiv which investigates related questions and proves a variant of Theorem \ref{t:stretching} for sets $A$ of sufficiently small measure (Proposition 11, \cite{Kneuss}), as well as much more. Their proof uses a covering lemma based on a result of \cite{ACP04}, which is no longer true in higher dimensions. Our proof is different and we believe can be adapted to work in higher dimensions as well. Moreover, the stretching result only for sufficiently small sets is not sufficient for our purposes in proving rough isometry. Our work was independent of \cite{Kneuss}.

The question considered in this paper has connections to several different problems in analysis, an interested reader is referred to \cite{Kneuss} and the references therein for more details.

\subsection{Proofs of Theorem \ref{t:stretching} and Theorem \ref{t:riuse}}
In this subsection we establish Theorem \ref{t:stretching} from Theorem \ref{t:stretchingfixed} and Theorem \ref{t:riuse} as a consequence of Theorem \ref{t:stretching}. We start with the easy proof of Theorem \ref{t:stretching} assuming Theroem \ref{t:stretchingfixed}.

\begin{proof}[Proof of Theorem \ref{t:stretching}]
Fix $0<\gamma < 1-\gamma' <1$. Fix a set $A$ with $\lambda(A)=\gamma$. Define a sequence of bijections $\psi_{i}$ on $[0,1]^2$ and a sequence of Borel sets $A_{i}$ as follows. Set $A=A_1$. If $\lambda(A_i)< 1-\gamma'$, then define $\psi_{i}=\phi_{A_i}$ where $\phi_{A_i}$ is given by Theorem \ref{t:stretchingfixed}. If $\lambda(A_{i})\geq 1-\gamma'$, set $\psi_{i}$ to be the identity map. Set $A_{i+1}=\psi_{i}(A_i)$. Fix $\eta>0$. Set $n_0=\lceil \frac{1-\gamma' -\gamma}{\varepsilon(\gamma,\gamma',\eta)}+1\rceil$. It follows from Theorem \ref{t:stretchingfixed} that the function $\phi_0=\psi_{n_0}\circ \cdots \circ \psi_1$ satisfies all the conditions of Theorem \ref{t:stretching} with $C_0= (1+\eta)^{n_0}$. This completes the proof.
\end{proof}

Next we show how Theorem \ref{t:riuse} follows from Theorem \ref{t:stretching}.

\begin{proof}[Proof of Theorem \ref{t:riuse}]
Fix $X$ such that $k_X(n)\geq \delta n^2$. Let $A$ denote the union of unit squares in $[0,n]^2$ that contain points of $X$. Now fix $\kappa>0$ such that $(1-e^{-\kappa ^2})^{1/\kappa^2}> e^{-\epsilon}$ and $\delta'>0$ such that $(1-e^{-\kappa ^2})^{(1-\delta')\kappa^{-2}}e^{-\delta'}\geq e^{-\epsilon}$. Since $A$ has measure at least $\delta n^2$, by Theorem \ref{t:stretching} there exists a $C(\delta,\delta')$ bi-Lipschitz map $\phi$ from $[0,n]^2$ to $[0,n]^2$ which is identity on the boundary and such that $\lambda(\phi(A))\geq (1-\delta')n^2$ (if $\lambda(A)\geq (1-\delta')n^2$, we take $\phi$ to be the identity map). Now let $B$ denote the union of squares in $[0,n]^2$ of the form $[\kappa j, \kappa(j+1)]\times [\kappa \ell, \kappa(\ell+1)]$ that intersect $\phi(A)$, and let $k'_Y(n)$ denote their numbers. Clearly $k'_Y(n)\geq (1-\delta')n^2\kappa^{-2}$. Now let $\mathcal{E}$ denote the event that each square of the form $[\kappa j, \kappa(j+1)]\times [\kappa \ell, \kappa(\ell+1)]$ contained in $B$ contains at least one point of $Y$ and there are no points of $Y$ in $[0,n]^2$ outside $B$. It is easy to see that on $\mathcal{E}$, there exists $M=M(\delta,\epsilon)$, $D=D(\delta, \epsilon)$, $C=C(\delta,\epsilon)$ such that we have $X_n\hookrightarrow_{(M,D,C)} Y_n$. Hence

$$\P[X_n\hookrightarrow_{(M,D,C)} Y_n\mid X_n]\geq \P[\mathcal{E}\mid X] \geq (1-e^{-\kappa^2})^{(1-\delta')n^2\kappa ^{-2}}e^{-\delta' n^2}\geq e^{-\epsilon n^2}. $$

This completes the proof.
\end{proof}

Rest of this paper is devoted to proving Theorem \ref{t:stretchingfixed}. From now on, $\gamma$ and $\gamma'$ and $\eta$ will be fixed positive numbers such that $0<\gamma<1-\gamma' <1$. Also we shall fix a Borel set $A\subseteq [0,1]^2$ with $\lambda(A)\in [\gamma, 1-\gamma']$.

\subsection{An Overview of the Proof of Theorem \ref{t:stretchingfixed}}
To prove Theorem \ref{t:stretchingfixed} one needs to construct a map satisfying the required conditions which expands the regions where the set $A$ has higher density and compress the regions where set $A$ has lower density. For example it is not hard to see that one can construct such a function if the set $A$ is contained in, say, the left half of the unit square (i.e., $[0,\frac{1}{2}]\times [0,1]$) then the conclusion of Theorem \ref{t:stretchingfixed} holds. Similar construction works for sets which have different densities in the left half and the right half of $[0,1]^2$, (see \S~\ref{s:aux}). For more complicated sets we use the same idea recursively at different scales.

In \S~\ref{s:aux}, we construct an auxiliary family of bijections $\{\Psi_{\delta}\}_{\delta\in (-1,1)}$ from $[0,1]^2$ onto itself which are identity on the boundary and which stretches all regions within the left half of $[0,1]^2$ by the same factor $1+\delta$.  The functions $\Psi_{\delta}$ also satisfy certain regularity conditions, in particular these are bi-Lipschitz functions with Lipschitz constant $1+O(\delta)$. We divide the unit square into dyadic squares and rectangles at different scales recursively such that each square (rectangle) at a level is divided in two halves by the rectangles (squares) in the next level (see \S~\ref{s:dyadic}). We use the auxiliary functions $\Psi_{\delta}$ to construct bijections on the dyadic squares (rectangles) at different levels which stretches each half of these dyadic squares (which are dyadic rectangles at the next scale) proportionally to the density of $A$ in these dyadic rectangles. Finally we compose these functions (see \S~\ref{s:mart}) up to a large number of levels (with the number of levels depending on the set $A$ and the location of the dyadic square) to obtain the required stretching function.  However, the control on the Lipschitz constant worsens with the number of compositions and it is necessary to establish that one can maintain a uniform control over the Lipschitz constant.  While the construction is deterministic, our proof makes use of a probabilistic analysis.

We let $X$ denote a uniformly chosen point in $[0,1]^2$. Let us reveal sequentially for $n\geq 1$, which dyadic square at level $n$ it belongs to. The expected density of $X$ is then a martingale. We make use of this martingale to analyse the function described above. Roughly we show that there exists a stopping time $\tau$ such that if we compose the stretching functions up to level $\tau$, then the area of $A$ is increased by $\epsilon$, however the bi-Lipschitz constant is still controlled by $1+\eta$. This argument is spanned over \S~\ref{s:stop}, \S~\ref{s:estuv}, \S~\ref{s:esty} and finally we complete the proof of Theorem \ref{t:stretchingfixed} in \S~\ref{s:final}.

{\bf A word about notation: parameters and constants}
In the course of the proof of Theorem \ref{t:stretchingfixed} over the next few sections, we shall have occasion of using many constants and parameters. By an absolute constant we shall mean a constant that depends  only on $\gamma$ and $\gamma'$. We shall denote by $C,c$ absolute constants whose values may vary through the proof while numbered constants $C_1,C_2,\ldots,$  and  $\varepsilon_1$, $\varepsilon_2,\ldots,$ will denote fixed constants whose values are fixed throughout the paper and in particular are independent of the set $A$.  When we use a matrix norm for a matrix $\mathbf{M}$, unless otherwise stated, $||\mathbf{M}||$ will denote its $\ell_{\infty}$ norm, i.e., the maximum of absolute values of its entries.


\section{The Stretching Function}
\label{s:aux}
To construct the map $\phi$ we shall need an auxiliary stretching map $\Psi_{\delta}$ where $\delta\in (-1,1)$ is a stretching parameter. $\Psi_{\delta}$ will be a bijection on $[0,1]^2$ which is identity on the boundary of $[0,1]^2$. We now move towards the construction of $\Psi_{\delta}$.

%
%
%
%
%

\subsection{Construction of $\Psi_{\delta}$}
For the rest of this subsection, fix $\delta\in (-1,1)$. We first construction the following parametrisation.

\subsection*{Parametrisation}
Let $h: [0,\frac{1}{2})\rightarrow [0,\infty)$ be a function with the following properties.

\begin{enumerate}
\item $h(0)=0$, and $h(r)=0$ for all $r\leq \frac{1}{10}$.
\item $h(r)\rightarrow \infty$ as $r\rightarrow \frac{1}{2}$.
\item $h$ is (weakly) increasing.
\item $h$ is thrice continuously differentiable with $(h'(r))^2=O((r+h(r))^3)$, $h''(r)h'(r)=O((r+h(r))^3)$ and $h^{(3)}(r)=O((r+h(r))^2)$.
\end{enumerate}

It is easy to see that such an $h$ exists, e.g., we could take a function that behaves like $e^{-1/(r-1/10)}$ near $\frac{1}{10}$ and like $(r-\frac{1}{2})^{-6}$ near $\frac{1}{2}$. Fix such an $h$ for the rest of this section.

Clearly, there is a unique $r_0\in (0,\frac{1}{2})$ such that
$$ 2r_0h(r_0)+r_0^2=\frac{1}{4}.$$

Define $\Theta(r)$ as follows.
\begin{equation*}
\Theta(r)=
\begin{cases}
\arccos \frac{h(r)}{h(r)+r}~{\rm{if}}~r\leq r_0,\\
\arcsin \frac{1}{2(h(r)+r)}~{\rm{if}}~ r>r_0.
\end{cases}
\end{equation*}

\subsection*{The $r-\theta$ parametrisation}
Consider the bijection $K:[0,1]\times [0,\frac{1}{2})\setminus \{(1/2,0)\} \rightarrow (0,\frac{1}{2})\times (-1,1)$ defined as follows. We have $(x_1,x_2)\mapsto (r,\theta)$ defined by

$$x_1= \frac{1}{2}+(r+h(r))\sin (\theta\Theta(r)),x_2= (r+h(r))\cos (\theta \Theta(r))-h(r).$$

We shall work with this parametrisation in the lower half of the unit square. The level lines of the function $r$ and their  reflections about the line $x_2=\frac{1}{2}$ is shown in Figure \ref{f:flowlines}.

\begin{figure*}[h]
\begin{center}
\includegraphics[height=8cm,width=10cm]{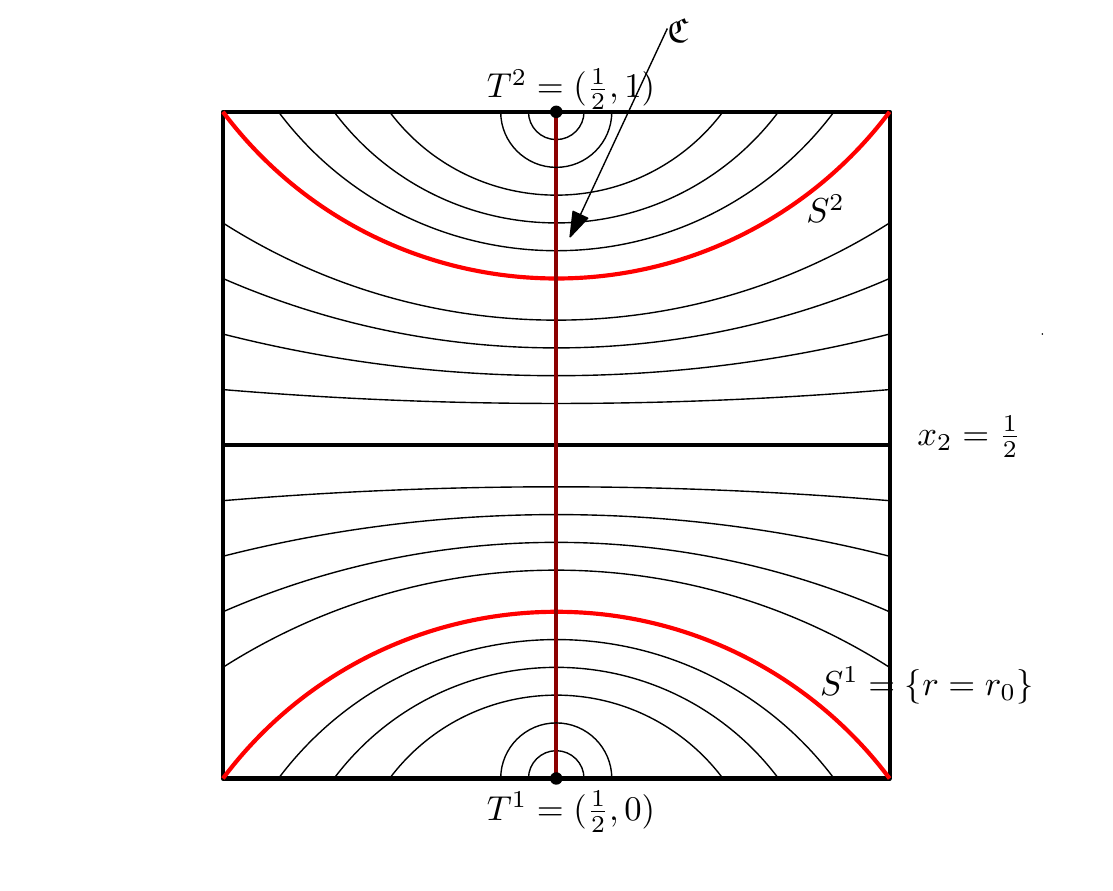}
\caption{Level lines of the function $r$ in the $r-\theta$ parametrisation and their reflections about the line $x_2=\frac{1}{2}$}
\label{f:flowlines}
\end{center}
\end{figure*}

Notice that this transformation is $C^1$ except on $\{r=r_0\}$ and the Jacobian matrix for the transformation $K^{-1}$ is given by

\[
J(r,\theta)=\left[\begin{array}{cc}
(1+h'(r))\sin (\theta \Theta(r))+(r+h(r))\Theta'(r)\theta\cos(\theta\Theta(r))& (r+h(r))\Theta(r)\cos (\theta\Theta(r))\\
(1+h'(r))\cos (\theta \Theta(r))-(r+h(r))\Theta'(r)\theta\sin(\theta\Theta(r))-h'(r)& -(r+h(r))\Theta(r)\sin(\theta\Theta(r))
\end{array}
\right].
\]

The (absolute value) of the determinant of the Jacobian is given by

$$\Theta(r)(r+h(r))(1+h'(r)-h'(r)\cos(\theta\Theta(r))).$$

\subsection*{Constructing $\Psi_{\delta}$}
Let $\delta\in (-1,1)$ be fixed. For $r\in (0,\frac{1}{2})$, define the function $g_{r,\delta}=g_{r}:[-1,1]\rightarrow [-1,1]$ with the following properties.

\begin{enumerate}
\item $g_r$ is an increasing bijection with $g(-1)=-1$ and $g(1)=1$.
\item For each $\ell\in [-1,0]$, we have

\begin{equation}
\label{e:grtheta1}
(1+\delta) \int_{-1}^{\ell} \left(1+h'(r)-h'(r)\cos(\theta\Theta(r))\right)~d\theta  = \int_{-1}^{g_{r}(\ell)} \left(1+h'(r)-h'(r)\cos(\theta\Theta(r))\right)~d\theta .
\end{equation}

\item For each $\ell\in [0,1]$ we have
\begin{equation}
\label{e:grtheta2}
(1-\delta) \int_{\ell}^{1} \left(1+h'(r)-h'(r)\cos(\theta\Theta(r))\right)~d\theta = \int_{g_{r}(\ell)}^{1} (1+h'(r)-h'(r)\cos(\theta\Theta(r)))~d\theta .
\end{equation}
\end{enumerate}

\begin{figure*}[h]
\begin{center}
\includegraphics[height=8cm,width=16cm]{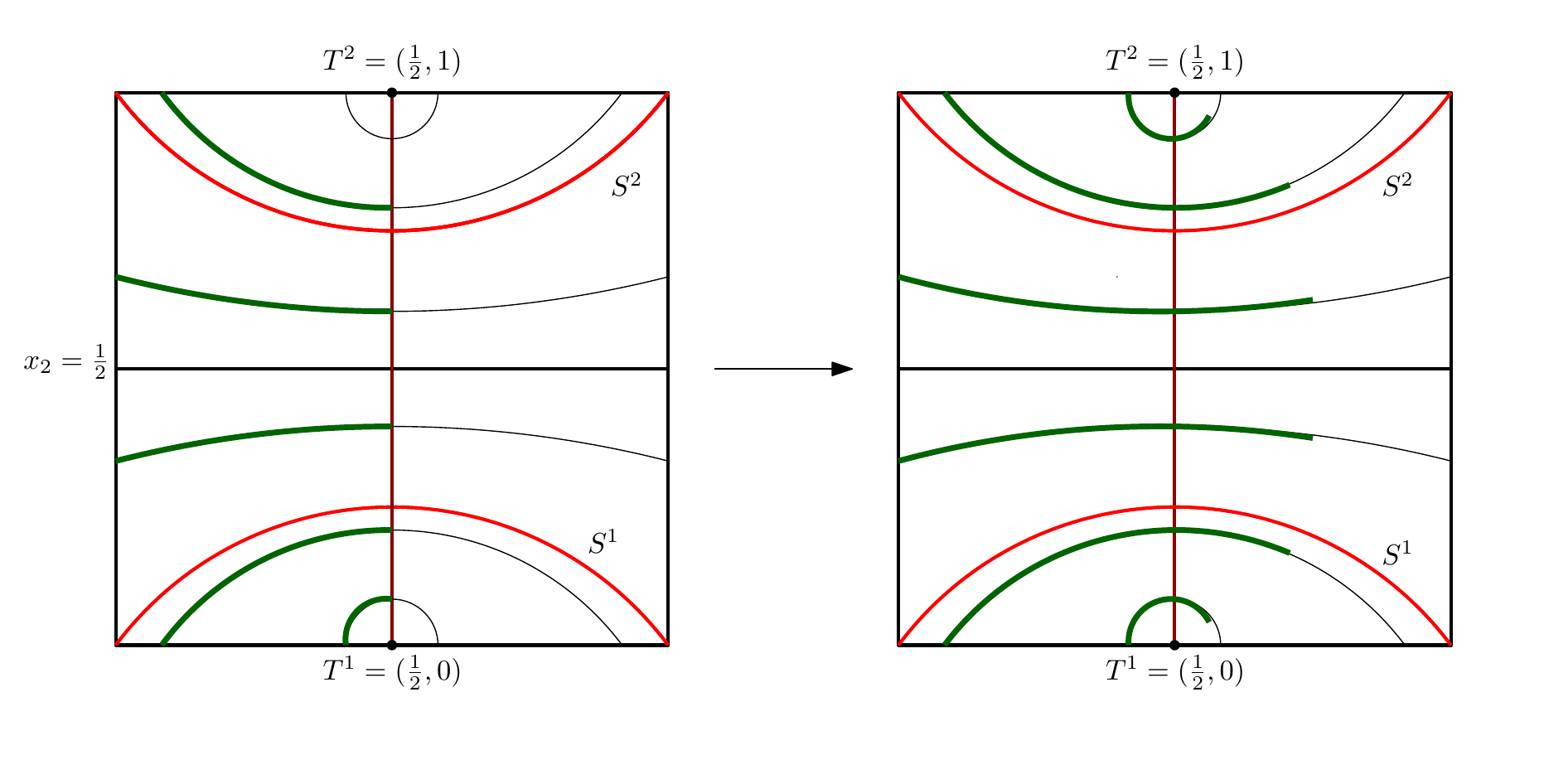}
\caption{For $\delta>0$, the functions $g_r$ stretch the left half of the level lines of the function $r$}
\label{f:psidelta}
\end{center}
\end{figure*}

That such a function exists and is unique follows from the facts that for each $r\in (0,\frac{1}{2})$ the integrand
$1+h'(r)-h'(r)\cos(\theta\Theta(r))$ is strictly positive on $[-1,1]$, is invariant under the transformation $\theta\mapsto -\theta$ and the hypothesis that $\delta\in (-1,1)$. Now define the bijection $H_{\delta}=H: (0,\frac{1}{2})\times [-1,1]\rightarrow [0,\frac{1}{2})\times [-1,1]$ defined by $(r,\theta)\mapsto (r,g_r(\theta))$. Define $\Psi_{\delta}$ on $[0,1]\times [0,\frac{1}{2})\setminus \{(1/2,0)\}$ by
$$\Psi_{\delta}(x_1,x_2)=(\Psi_{\delta}^{1}(x_1,x_2),\Psi_{\delta}^{2}(x_1,x_2))=K^{-1}(H_{\delta}(K(x_1,x_2))).$$
We define $\Psi_{\delta}(1/2,0)=(1/2,0)$ and extend $\Psi_{\delta}$ to $[0,1]^{2}$ in the following way. For $(x_1,x_2)\in [0,1]^2$ with $x_2>\frac{1}{2}$ define
$$\Psi_{\delta}(x_1,x_2)=(\Psi_{\delta}^{1}(x_1,x_2),\Psi_{\delta}^{2}(x_1,x_2))=(\Psi_{\delta}^{1}(x_1,1-x_2),1-\Psi_{\delta}^{2}(x_1,1-x_2)).$$
On the line $x_2=\frac{1}{2}$, we set
\begin{equation*}
\Psi_{\delta}(x_1,\frac{1}{2})=(\Psi_{\delta}^{1}(x_1,\frac{1}{2}),\Psi_{\delta}^{2}(x_1,\frac{1}{2}))=
\begin{cases}
(x_1(1+\delta),\frac{1}{2})~\text{for}~x_1\leq \frac{1}{2},\\
(\frac{1+\delta}{2}+(1-\delta)(x_1-\frac{1}{2}),\frac{1}{2})~\text{for}~x_1\geq \frac{1}{2}.
\end{cases}
\end{equation*}

\subsection{Basic properties of $\Psi_{\delta}$}
Over the next few lemmas we list useful properties of the function $\Psi_{\delta}$ as constructed above. The next lemma is immediate and we omit the proof.

\begin{lemma}
\label{l:psideltabasic}
For each $\delta\in (-1,1)$, $\Psi_{\delta}$ as constructed above is a bijection from $[0,1]^2$ onto itself and $\Psi_{\delta}(x)=x$ for all $x\in \partial [0,1]^2$,
\end{lemma}

\begin{lemma}
\label{l:psideltacont}
$\Psi_{\delta}$ as defined above is continuous on $[0,1]^2$.
\end{lemma}

Proof of this lemma is deferred to \ref{s:appa}.  The next lemma shows that the left and rights sides are stretched uniformly by ratios of $1+\delta$ and $1-\delta$ respectively.

\begin{lemma}
\label{l:psideltaarea}
For each $\delta\in (-1,1)$, $\Psi_{\delta}$ defined as above it satisfies the following properties.

\begin{enumerate}
\item[(i)] For $\Lambda_{L}=[0,\frac{1}{2}]\times[0,1]$, we have $\lambda(\Psi_{\delta}(\Lambda_{L}))=(1+\delta)\lambda(\Lambda_{L})$.

\item[(ii)] For $\Lambda_{L}$ as above and $\Lambda_{R}=[\frac{1}{2},1]\times [0,1]$, we have for $i=L,R$, and for $B\subseteq \Lambda_{i}$, $B$ measurable,
$$\lambda(B)=(1+\delta)\lambda(B \cap \Lambda_{L}) + (1-\delta)\lambda(B \cap \Lambda_{R}).$$
\end{enumerate}
\end{lemma}

\begin{proof}
We first prove that for all $B\subseteq \Lambda_{L}$, $\lambda(\Psi_{\delta}(B))=(1+\delta)\lambda(B)$. Fix $B\subseteq \Lambda_{L}$. Without loss of generality assume $B\subseteq \tilde{\Lambda}^1$ as well. We have using (\ref{e:grtheta1}),
\begin{eqnarray*}
(1+\delta)\lambda(B)&=&(1+\delta)\int_{0}^{1/2}\int_{0}^{1/2}1_{B}~dx~dy\\
&=&(1+\delta)\int_{K(B)}(r+h(r))(1+h'(r)-h'(r)\cos(\theta\Theta(r)))\Theta(r)~dr~d\theta\\
&=& (1+\delta)\int_{0}^{1/2}(r+h(r))\Theta(r)\left(\int_{\theta\in K(B)_{r}}(1+h'(r)-h'(r)\cos(\theta\Theta(r)))~d\theta\right)~dr\\
&=&\int_{0}^{1/2}(r+h(r))\Theta(r)\left(\int_{\theta\in g_r(K(B)_{r})}(1+h'(r)-h'(r)\cos(\theta\Theta(r)))~d\theta\right)~dr\\
&=&\int_{H(K(B))}(r+h(r))(1+h'(r)-h'(r)\cos(\theta\Theta(r)))\Theta(r)~dr~d\theta\\
&=&\lambda(\Psi_{\delta}(B)).
\end{eqnarray*}

Similarly it can be shown using (\ref{e:grtheta2}) that for all $B\subseteq \Lambda_R$, we have $\lambda(\Psi_{\delta}(B))=(1-\delta)\lambda(B)$. This completes the proof of the lemma.
\end{proof}

\begin{figure*}[h]
\begin{center}
\includegraphics[height=12 cm,width=12 cm]{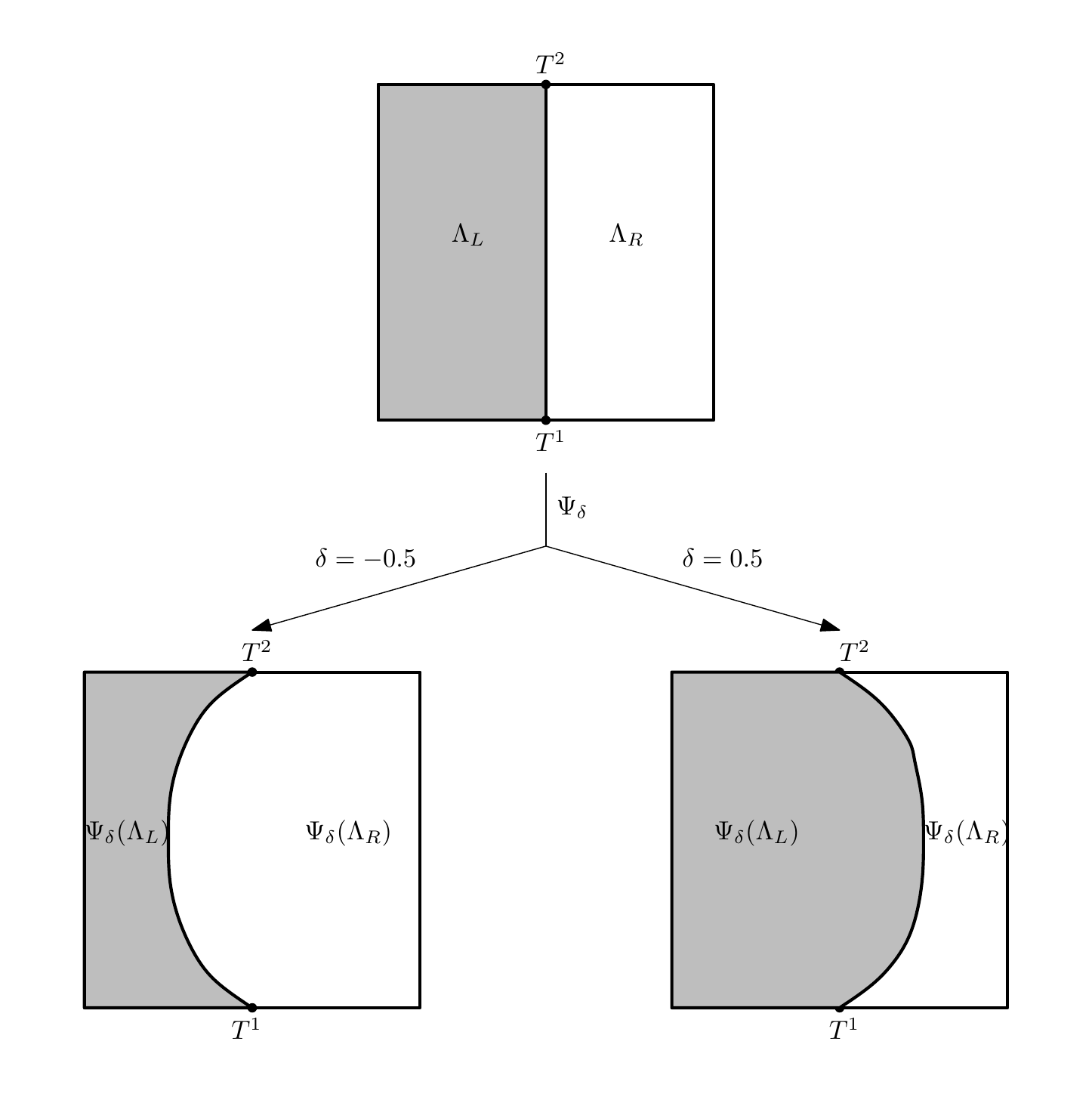}
\caption{$\Psi_{\delta}$ for different values of $\delta$}
\label{f:psideltastretch}
\end{center}
\end{figure*}

\subsection{Smoothness of $\Psi_{\delta}$}

Now we need to establish that $\Psi_{\delta}$ has certain smoothness properties.

\subsubsection{Crack, Twists and Seams: Geometric definitions}
We introduce the following geometric definitions for $[0,1]^2$.
\begin{definition}[Crack and Twists]
\label{d:crack}
The line $x_1=\frac{1}{2}$ is called the \emph{crack} $\mathfrak{C}$ in the unit square $[0,1]^2$. Let $T^1$ and $T^2$ denote the points where the \emph{crack} intersects the boundary, $T^1$ and $T^2$ are called \emph{twists} in $[0,1]^2$. Often we shall call $T=T^1\cup T^2$ as \emph{twists} in $[0,1]^2$. For $r_1<\frac{1}{10}$ small, define $T_r^{1}=K^{-1}(\{r<r_1\})$ and define $T_r^{2}$ to be the reflection of $T_r^{1}$ on the line $x_2=\frac{1}{2}$. We call $T_r=T_r^1\cup T_r^2$ as the \emph{blown up twists} of $[0,1]^2$.
\end{definition}

\begin{definition}[Seams]
\label{d:seams}
Let $S^1=K^{-1}(\{r=r_0\})$ in the parametrisation described above. Let $S^2$ be the reflection of $S_1$ on the line $x_2=1/2$. We shall call $S^1$ and $S^2$ (or, $S=S^1\cup S^2$) \emph{seams} of $[0,1]^2$.
\end{definition}

\subsubsection{Estimates for $\Psi_{\delta}'$}
\begin{proposition}
\label{p:psideltaderivative}
$\Psi_{\delta}$ is differentiable at all points in $[0,1]^2$ except possibly on the \emph{crack} $\mathfrak{C}$ and the \emph{seams} $S$. For $(x_1,x_2)\in [0,1]^2\setminus (\mathfrak{C}\cup S)$, let $\mathbf{J}_{\Psi_{\delta}}(x)$ denote the Jacobian matrix of the transformation $\Psi_{\delta}$ evaluated at $x=(x_1,x_2)$. Then $\mathbf{J}_{\Psi_{\delta}}$ is continuous on $[0,1]^2\setminus (\mathfrak{C}\cup S$ and there exists an absolute constant $C_1$ (not depending on $\delta, x$, possibly depending on $h$) such that
$$||\mathbf{J}_{\Psi_{\delta}}(x)-{\rm{\mathbf I}}||\leq C_1\delta.$$
\end{proposition}

To prove this proposition, we shall need a few lemmas, dealing with the functions $g_r(\theta)$ and $J(r,\theta)$. These lemmas will be proved in \ref{s:appa}.

\begin{lemma}
\label{l:grthetasmoothnesstheta}
There exists an some absolute constant $C>0$ such that the following hold.

\begin{enumerate}
\item[(i)] $(g_r(\ell)-\ell)\leq C\delta(\ell+1)$ for all $\ell\leq 0$.
\item[(ii)] We have $$\sup_{0<r<\frac{1}{2},\theta\in [-1,0)\cup (0,1]} |\frac{\partial g_{r}(\theta)}{\partial \theta}-1|\leq C\delta.$$
\item[(iii)] $(g_r(\ell)-\ell)=O(\frac{h'(r)}{(r+h(r))^2})$ as $r\rightarrow \frac{1}{2}$ uniformly for all $\ell\leq 0$.
\end{enumerate}
\end{lemma}

Proof of this lemma is provided in \ref{s:appa}.

\begin{lemma}
\label{l:grthetasmoothnessr}
There exists an absolute constant $C>0$ such that we have the following.

\begin{enumerate}
\item[(i)] We have
$$\sup_{0<r<\frac{1}{2}, r\neq r_0, \theta\in [-1,0)\cup (0,1]} |\frac{\partial g_{r}(\theta)}{\partial r}|\leq C\delta.$$
\item[(ii)] We have $|\frac{\partial g_{r}(\theta)}{\partial r}|\leq \delta O\left(\frac{h''(r)}{(r+h(r))^2}+\frac{h'(r)^3}{(r+h(r))^5}\right)$ as $r\rightarrow \frac{1}{2}$.
\end{enumerate}
\end{lemma}

The proof is deferred to \ref{s:appa}.

\begin{lemma}
\label{l:jrtheta}
Let $\mathbf{M}(r,\theta)=J(r,g_r(\theta))-J(r,\theta)$. Then there exists an absolute constant $C$ such that $||\mathbf{M}(r,\theta)||\leq C\delta$.
\end{lemma}

\begin{proof}
This follows immediately from the formula for $J(r,\theta)$ and part $(i)$ of Lemma  \ref{l:grthetasmoothnesstheta}.
\end{proof}

Now we are ready to prove Proposition \ref{p:psideltaderivative}.

\begin{proof}[Proof of Proposition \ref{p:psideltaderivative}]
By symmetry it is enough to consider $x=(x_1,x_2)$ such that $x_2\leq \frac{1}{2}$.
To start with, we assume $x=(x_1,x_2)$ with $x_2<\frac{1}{2}$. The differentiability is easy to establish. Let $(r,\theta)=K(x_1,x_2)$. Let $\mathbf{J}_{H}(r,\theta)$ denote the Jacobian matrix of the transformation $H$ evaluated at the point $(r,\theta)$. By Chain rule, it follows that

$$\mathbf{J}_{\Psi_{\delta}}(x)= J(r,\theta)^{-1}\mathbf{J}_{H}(r,\theta)J(r,g_r(\theta)).$$

It follows that
\begin{equation}
\label{e:jpsix}
\mathbf{J}_{\Psi_{\delta}}(x)-{\rm{\mathbf I}}= J(r,\theta)^{-1}(\mathbf{J}_{H}(r,\theta)-{\rm{\mathbf I}})J(r,g_r(\theta))+J(r,\theta)^{-1}\mathbf{M}(r,\theta).
\end{equation}

It follows from the definition of $H$ that
\[
\mathbf{J}_{H}(r,\theta)=\left[
\begin{array}{cc}
1& 0\\
\frac{\partial g_{r}(\theta)}{\partial r}& \frac{\partial g_{r}(\theta)}{\partial \theta}
\end{array}
\right].
\]

It follows from Lemma \ref{l:grthetasmoothnesstheta} and Lemma \ref{l:grthetasmoothnessr} that there exists an absolute constant $C$ that $$||\mathbf{J}_{H}(r,\theta)-{\rm{\mathbf I}}||\leq C\delta.$$

Observe the following. Fix $r_0>\frac{1}{10}>r_1>0$. Observe that there exists a constant $C$ such that

\begin{equation}
\label{e:jrtheta2}
\sup_{r_1<r<\frac{1}{2},\theta} ||J(r,\theta)||\vee ||J(r,\theta)^{-1}||\leq C.
\end{equation}

Hence it follows from Lemma \ref{l:jrtheta} and (\ref{e:jpsix}) that there exists and absolute constant $C$ such that for all $x\in K^{-1}(\{r_1<r<\frac{1}{2}\})$ we have

$$||\mathbf{J}_{\Psi_{\delta}}(x)-\rm{\mathbf{I}}||\leq C\delta.$$

Since $h(r)=0$ for each $r\leq r_1$ it follows that for each $r<r_1$, and $\theta\leq 0$, we have $g_r(\theta)=-1+(1+\delta)(\theta+1)$ and it can be verified by direct computation that there exists $C>0$ such that

%
%
%
%
%
%

$$\sup_{0<r\leq r_1} ||J(r,\theta)^{-1}(\mathbf{J}_{H}(r,\theta)-{\rm{\mathbf I}})J(r,g_r(\theta))+J(r,\theta)^{-1}\mathbf{M}(r,\theta)||\leq C\delta.$$

It is also easy to see using symmetry of construction and Lemma \ref{l:grthetasmoothnessr} that $\mathbf{J}_{\Psi_{\delta}}$ is continuous on $\{x_2=\frac{1}{2}\}\setminus \{(1/2,1/2)\}$.

%
%
%
%

This completes the proof of the proposition by choosing $C_1$ appropriately.
\end{proof}

Finally we have the following.

\begin{proposition}
\label{p:lipschitz}
Let $1>\chi>0$ be fixed. Then for all $\delta<\delta_0=\frac{\chi}{100(C_1+1)}$, $\Psi_{\delta}$ is bi-Lipschitz with Lipschitz constant $(1+\chi)$.
\end{proposition}

\begin{proof}
For $\delta< \frac{\chi}{100C_1}\wedge \frac{1}{100}$, it follows from Proposition \ref{p:psideltaderivative} that

$$\max\{||\mathbf{J}_{\Psi_{\delta}}-\rm{\mathbf{I}}||, ||\mathbf{J}_{\Psi_{\delta}}^{-1}-\rm{\mathbf{I}}||\}\leq \frac{\chi}{2}.$$

The proposition follows.
\end{proof}

\begin{proposition}
\label{p:distance}
For $x\in [0,1]^2$, let $\tilde{d}(x)$ denotes its distance from the corners of $[0,1]^2$. Then there exists an absolute constant $C>0$ such that for all $x\in [0,1]^2$ we have

$$|\Psi_{\delta}(x)-x|\leq C\delta\tilde{d}(x).$$
\end{proposition}

\begin{proof}
This follows easily from the definition of $\Psi_{\delta}$ and part $(i)$ of Lemma \ref{l:grthetasmoothnesstheta}.
\end{proof}

\subsubsection{Estimates for $\Psi_{\delta}''$}

We want to show that the second derivative of $\Psi_{\delta}$ remains bounded away from the \emph{Twists}.
\begin{proposition}
\label{p:psidelta2ndderivative}
$\Psi_{\delta}$ is twice differentiable at all points in $[0,1]^2$ except possibly on the \emph{crack} $\mathfrak{C}$ and the \emph{seams} $S$. Then there exists an absolute constant $C_2>0$ (not depending on $\delta, x$, possibly depending on $h$) such that for $(x_1,x_2)\in [0,1]^2\setminus (\mathfrak{C}\cup S \cup T_{r_1/2})$, we have
$$||\Psi_{\delta}''(x_1,x_2)||\leq C_2\delta.$$
\end{proposition}

\begin{proof}
Without loss of generality consider $x=(x_1,x_2)$ with $x_1<\frac{1}{2}$ and $x_2<\frac{1}{2}$. Let $(r,\theta)=K(x_1,x_2)$. Let us denote

$$\mathbf{H}(r,\theta)= J(r,\theta)^{-1}(\mathbf{J}_{H}(r,\theta)-{\rm \mathbf{I}})J(r,g_r(\theta)).$$

For the rest of this subsection, let us introduce the following piece of notation. For a matrix $\mathbf{A}$, $\mathbf{A}_{r}$ (resp. $\mathbf{A}_{\theta}$) shall denote the entrywise derivative w.r.t. $r$ (resp. $\theta$) of the matrix $A$. Because of (\ref{e:jrtheta2}), it suffices to prove that for some absolute constant $C$ we have

$$||\mathbf{H}_r(r,\theta)||\leq C\delta, ||\mathbf{H}_{\theta}(r,\theta)||\leq C\delta.$$

It follows now from Lemma \ref{l:jrtheta} and Propostion \ref{p:psideltaderivative} that it suffices to prove the following.

\begin{enumerate}
\item[(i)] Let us denote $\tilde{\mathbf{J}}(r,\theta)=J(r,\theta)^{-1}$. There exists an absolute constant $C>0$ such that $||\tilde{\mathbf{J}}_{r}(r,\theta)||\leq C,||\tilde{\mathbf{J}}_{\theta}(r,\theta)||\leq C$.
\item[(ii)] Let $\mathbf{J}^{0}(r,\theta)=(\mathbf{J}_H(r,\theta)-{\rm{\mathbf{I}}})$. Then there exists an absolute constant $C$ such that  $||\mathbf{J}^{0}_r(r,\theta)||\leq C\delta, ||\mathbf{J}^{0}_{\theta}(r,\theta)||\leq C\delta$.
\item[(iii)] There exists a constant $C>0$ such that $||J_{r}(r,g_{r}(\theta))||\leq C, ||J_{\theta}(r,\theta)||\leq C\delta$.
\end{enumerate}

The above three assertions are proved below in Lemma \ref{l:2ndd1}, Lemma \ref{l:2ndd2} and Lemma \ref{l:2ndd3} respectively. This completes the proof of the proposition.
\end{proof}

\begin{lemma}
\label{l:2ndd1}
Let $\tilde{\mathbf{J}}_{r}(r,\theta)$ be defined as in the proof of Proposition \ref{p:psidelta2ndderivative}. Then there exists an absolute constant $C$ such that for $r\in (r_1,\frac{1}{2})$ and $\theta\leq 0$, we have $||\tilde{\mathbf{J}}_{r}(r,\theta)||\leq C,||\tilde{\mathbf{J}}_{\theta}(r,\theta)||\leq C$.
\end{lemma}

\begin{lemma}
\label{l:2ndd2}
Let $\mathbf{J}^0$ be defined as in the proof of Proposition \ref{p:psidelta2ndderivative}. Then there exists an absolute constant $C$ such that  $||\mathbf{J}^{0}_r(r,\theta)||\leq C\delta, ||\mathbf{J}^{0}_{\theta}(r,\theta)||\leq C\delta$.
\end{lemma}

\begin{lemma}
\label{l:2ndd3}
There exists a constant $C>0$ such that $||J_{r}(r,g_{r}(\theta))||\leq C, ||J_{\theta}(r,\theta)||\leq C\delta$.
\end{lemma}

Proofs of the above three lemmas are deferred to \ref{s:appa}.

Finally we have the following proposition.

\begin{proposition}
\label{p:psiprimedeviation}
Let $x,x'\in [0,1]^2$ be such that $\Psi_{\delta}$ is differentiable at both $x$ and $x'$. Let $S_{x,x'}$ denote the event that the line joining $x$ and $x'$ intersects $S$. Also set $$g(x,x')=\min\{|x-(1/2,0)|,|x'-(1/2,0)|,|x-(1/2,1)|,|x'-(1/2,1)|\}.$$
 Then there exists an absolute constant $C_3>0$ such that we have
$$||\Psi_{\delta}'(x)-\Psi_{\delta}'(x')||\leq C_3\delta\left(\frac{|x-x'|}{g(x,x')}\wedge 1 +1_{S_{x,x'}}\right).$$
\end{proposition}

\begin{proof}
Clearly it follows from Proposition \ref{p:psideltaderivative} that for all $x,x'$, we have $||\Psi_{\delta}'(x)-\Psi_{\delta}'(x')||\leq 2C_1\delta$. For $x\in T_{r_1}$ it follows from explicit computations
that
$||\Psi_{\delta}''(x)|| \leq  \frac{C}{g(x,x)}$ for some absolute constant $C>0$. It follows from mean value theorem that for $x,x'\in K^{-1}(T_{r_1})$ we have
$$||\Psi_{\delta}'(x)-\Psi_{\delta}'(x')||\leq C\delta\left(\frac{|x-x'|}{g(x,x')}\right)$$
for some absolute constant $C$. Notice that the same also follows for $x,x'\in [0,1]^2\setminus T_{r_1}$ if $S_{x,x'}$ does not hold using Proposition \ref{p:psidelta2ndderivative}. Now consider the case $x\in T_{r_1}, x'\notin T_{r_1}$, $S_{x,x'}$ does not hold. If $x\notin T_{r_1/2}$, $g(x,x')> r_1/2$ and mean value theorem once again gives the result using Proposition \ref{p:psidelta2ndderivative}. In the only remaining case, $\frac{|x-x'|}{g(x,x')}>1$ and so there is nothing to check. All these combined proves the lemma for an appropriate choice of $C_3$.
\end{proof}

\subsection{Stretching Rectangles}

For the proof of Theorem \ref{t:stretching}, we shall need to stretch not only the unit square but also squares and rectangles of different sizes. Also we shall need to stretch rectangles not only along its length ($x_1$-direction) but also along its height ($x_2$-direction). We can do these by using $\Psi_{\delta}$ composed with some linear functions as follows.

For $u=(u_1,u_2)\in \R^2$, and $a,b>0$, let $D_{1,u,a,b}:[0,1]^2\rightarrow u+[0,a]\times [0,b]$ be the bijection given by $(x_1,x_2)\mapsto u+(ax_1,bx_2)$. Similarly, let $D_{2,u,a,b}:[0,1]^2\rightarrow u+[0,a]\times [0,b]$ be the bijection given by $(x_1,x_2)\mapsto u+(ax_2,bx_1)$. When $u,a,b$ are clear from the context we shall suppress the subscript $(u,a,b)$ and write $D_1$ or $D_2$ only.

Consider the rectangle $R=u+[0,a]\times [0,b]$. For $\Psi_{\delta}=(\Psi_{\delta}^{1}, \Psi_{\delta}^{2})$ constructed as above, the function $\Psi_{\delta,R,\rightarrow}:R\rightarrow R$ is a bijection defined by $\Psi_{\delta,R,\rightarrow}=D_1\circ \Psi_{\delta}\circ D_1^{-1}$. Similarly, the function $\Psi_{\delta,R,\uparrow}:R\rightarrow R$ is a bijection defined by $\Psi_{\delta,R,\uparrow}=D_2\circ \Psi_{\delta}\circ D_2^{-1}$. Note that $\Psi_{\delta,a,b,\rightarrow}$ stretches the rectangle in a left-right direction whereas $\Psi_{\delta,a,b,\uparrow}$ stretches the rectangle in an up-down direction.

\section{Dyadic Squares}
\label{s:dyadic}
For a rectangle $R$ in $\R^2$, whose sides are aligned with the coordinate axes (i.e., of the for $[x,x+a]\times [y,y+b]$), we call $a$ to be the \emph{length} of $R$, and $b$ to be the height of $R$. At each level $n\geq 0$ we write $[0,1]^2$ as a union of (not necessarily disjoint) rectangles aligned with the co-ordinate axes $\{\Lambda_{n}^{j}\}_{j=1}^{2^{n-1}}$ satisfying the following properties.

\begin{enumerate}
\item $\Lambda_1:=\Lambda_1^{1}=[0,1]^2$.
\item For a fixed $n$, for each $j\in [2^{n-1}]$, area of $\Lambda_{n}^{j}=2^{-(n-1)}$.
\item If $n$ is odd, then each $\Lambda_{n}^{j}$ is a square, i.e., has length and height equal. If $n$ is even, then for each $j\in [2^{n-1}]$, the height of $\Lambda_{n}^{j}$ is twice the length of $\Lambda_{n}^{j}$.
\item For each $n$ and $j$, $\Lambda_{n}^{j}=\Lambda_{n+1}^{2j-1}\cup \Lambda_{n+1}^{2j}$. For $n$ odd, $\Lambda_{n}^{j}$ and $\Lambda_{n+1}^{2j-1}$ has same height. For $n$ even, $\Lambda_{n}^{j}$ and $\Lambda_{n+1}^{2j-1}$ have same length.
\end{enumerate}

It is clear that there is a way to partition $\Lambda_1$ into rectangles at each level in such a manner, see Figure \ref{f:dyadic}. Suppose $u_{n,j}$ denote the top right corner of $\Lambda_{n,j}$. For definiteness we shall adopt the following convention. For each $j$, $u_{n,2j}\succeq u_{n,2j-1}$ where $\succeq$ denotes the lexicographic ordering on $\R^2$. It is clear that under such a convention there is a unique way to construct $\Lambda_{n}^{j}$s. We shall call the $\Lambda_{n,j}$ \emph{dyadic boxes} at level $n$.

\begin{figure*}[h]
\begin{center}
\includegraphics[height=12cm,width=16cm]{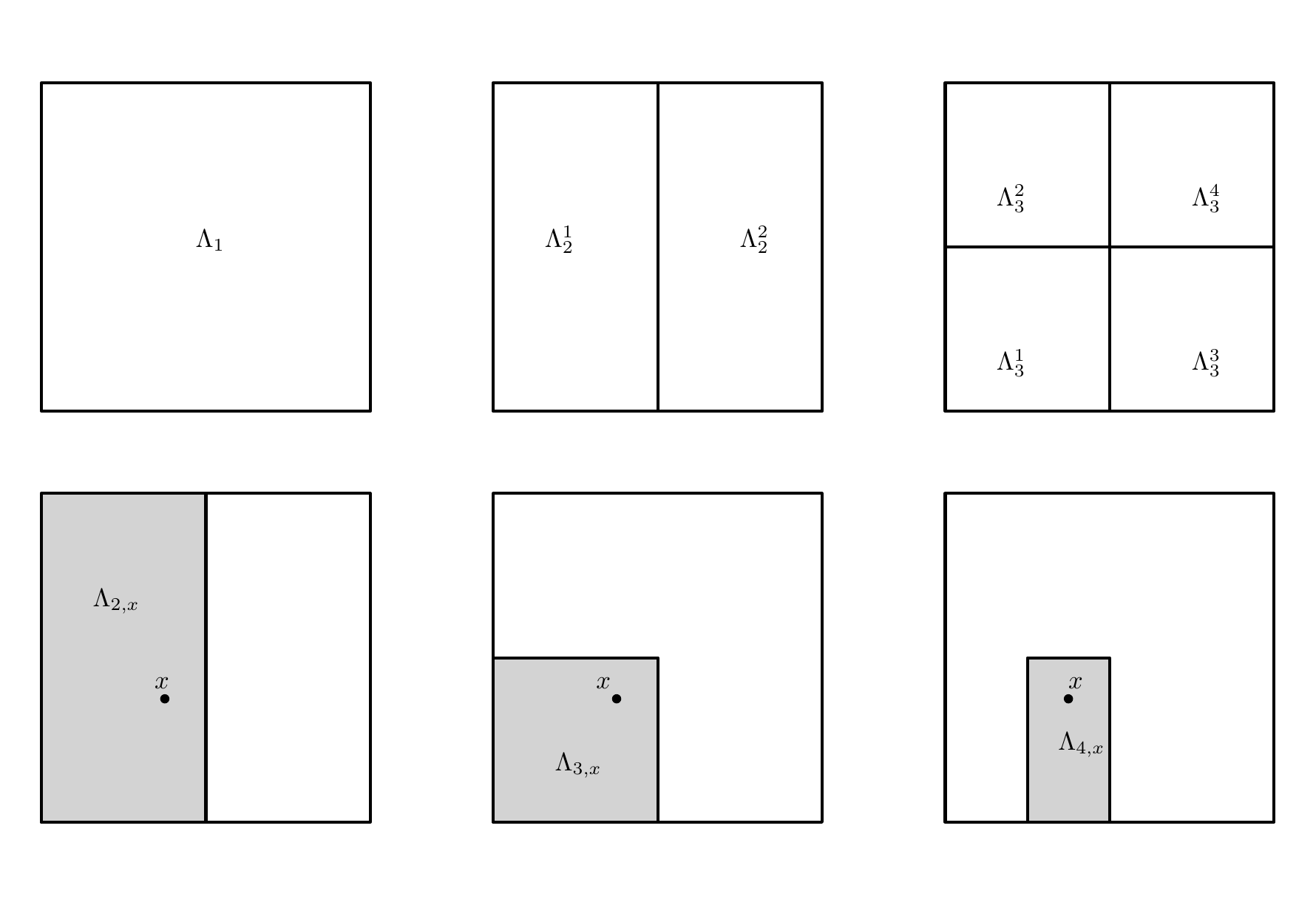}
\caption{Dyadic squares at different levels}
\label{f:dyadic}
\end{center}
\end{figure*}

 Let $$B_n=\{u\in \Lambda_1:u\in \Lambda_n^j\cap \Lambda_n^{j'}~\text{for some}~j\neq j'\}.$$
It is then clear from the construction that we have for $B:=\cup_n B_n$,  $\lambda(B)=0$.

For $x\in \Lambda_1\setminus B$, let us define $\Lambda_{n,x}=\Lambda_{n}^{j}$ where $j$ is such that $x\in \Lambda_{n}^{j}$. Let $\rho_{n,x}$ denote the density of $A$ in $\Lambda_{n,x}$, i.e.,
$$\rho_{n,x}=\frac{\lambda(A\cap \Lambda_{n,x})}{\lambda(\Lambda_{n,x})}.$$
Also, define $\Delta_{n,x}=\rho_{n+1,x}-\rho_{n,x}$. Notice that $|\Delta_{n,x}|$ is constant on $\Lambda_{n,x}\setminus B$.
We shall let  $v_{n,x}$ denote the bottom left corner of $\Lambda_{n,x}$. We shall denote by $L_1(n,x)$ and $L_2(n,x)$, the length and height of $\Lambda_{n,x}$ respectively. Also $\delta(n,x)=\frac{\Delta_{n,y}}{\rho_{n-1,x}}$ for $y\in \Lambda_{n,x}$ which is very close to $v_{n,x}$.

\subsection{Crack, Seams and Twists on Dyadic Boxes}
For $x\in [0,1]^2$, consider $\Lambda_{i,x}$, the $i$-th level dyadic box of $x$.  If $i$ is even, then define $S_{i,x}=D_{2,v_{i,x},L_1(i,x), L_2(i,x)}(S)$ to be the \emph{seams} of $\Lambda_{i,x}$, $\mathfrak{C}_{i,x}=D_{2}(\mathfrak{C})$ to be the \emph{crack} in $\Lambda_{i,x}$, and $T_{i,x}=D_{2}(T)$ to be the \emph{twists} in $\Lambda_{i,x}$. If $i$ is odd, we have same definitions except $D_2$ is replaced by $D_1$. Notice that, with these definition, it is clear that dyadic boxes at level $(i+1)$ are created by splitting level $i$ dyadic boxes in half along the cracks at level $i$.

\section{Stretching}
\label{s:mart}
\subsection{Martingales}
Now let $X$ be a random variable which is uniformly distributed on $\Lambda_1$. Notice that $\Lambda_{n,X}$, $\rho_{n,X}$, $\Delta_{n,X}$ are almost surely well-defined. Let $\cf_n$ denote the $\sigma$-algebra generated by $\Lambda_{n,X}$. The following observation is trivial.

\begin{observation}
\label{o:martingale}
We have $\rho_{n,X}=\P[X\in A\mid \cf_n]$ is a martingale with respect the filtration $\{\cf_n\}_{n\geq 1}$. Furthermore, $\rho_{n,X}\rightarrow I(X\in A)$ a.s. as $n\rightarrow\infty$.
\end{observation}

Clearly it follows from definitions that $|\Delta_{n,X}|$ is $\cf_{n}$ measurable. This leads to following easy and useful observation.

\begin{observation}
\label{o:martingalevariance}
Now consider a random time $\tau$ which is $\cf_n$-measurable. Then we have using the Optional Stopping Theorem
\begin{equation}
\label{e:varrho}
{\rm Var}[\rho_{\tau,X}]=\E\left[\sum_{i=2}^{\tau}\E[(\rho_{i,X}-\rho_{i-1,X})^2\mid \cf_{i-1}]\right]=\E \left(\sum_{i=1}^{\tau-1} \Delta_{i,X}^2\right).
\end{equation}
\end{observation}

\subsection{Stretching at different scales}
For $n=1,2,\ldots, $, define the function $\varphi_{n}$ on $\Lambda_{1}$ as follows. If $n$ is odd and $y\in \Lambda_{n,x}$, we define $$\varphi_{n}(y)=\Psi_{\delta(n,x),\Lambda_{n,x},\rightarrow}(y).$$
If $n$ is even and $y\in \Lambda_{n,x}$, we define $$\varphi_{n}(y)=\Psi_{\delta(n,x),\Lambda_{n,x},\uparrow}(y).$$
Clearly, each $\varphi_{n}$ is a bijection on $\Lambda_1$ that is identity on the boundary.

Define
\begin{equation}
\label{e:Phicomposed}
\Phi_{n}(x)=\varphi_1 \circ \ldots \circ \varphi_n(x).
\end{equation}

Clearly, $\Phi_{n}$ is also a bijection from $\Lambda_{1}$ to itself which is identity on the boundary. Also, set $\Phi_0$ to be the identity. In Figure \ref{f:phi1phi2}, we illustrate the sequence of functions $\Phi_0$, $\Phi_1$, $\Phi_2$, where $\delta(1,x)=0.5$ and $\delta(2,x)=-0.5$ on $\Lambda_2^1$.

\begin{figure*}[h]
\begin{center}
\includegraphics[height=6cm,width=16cm]{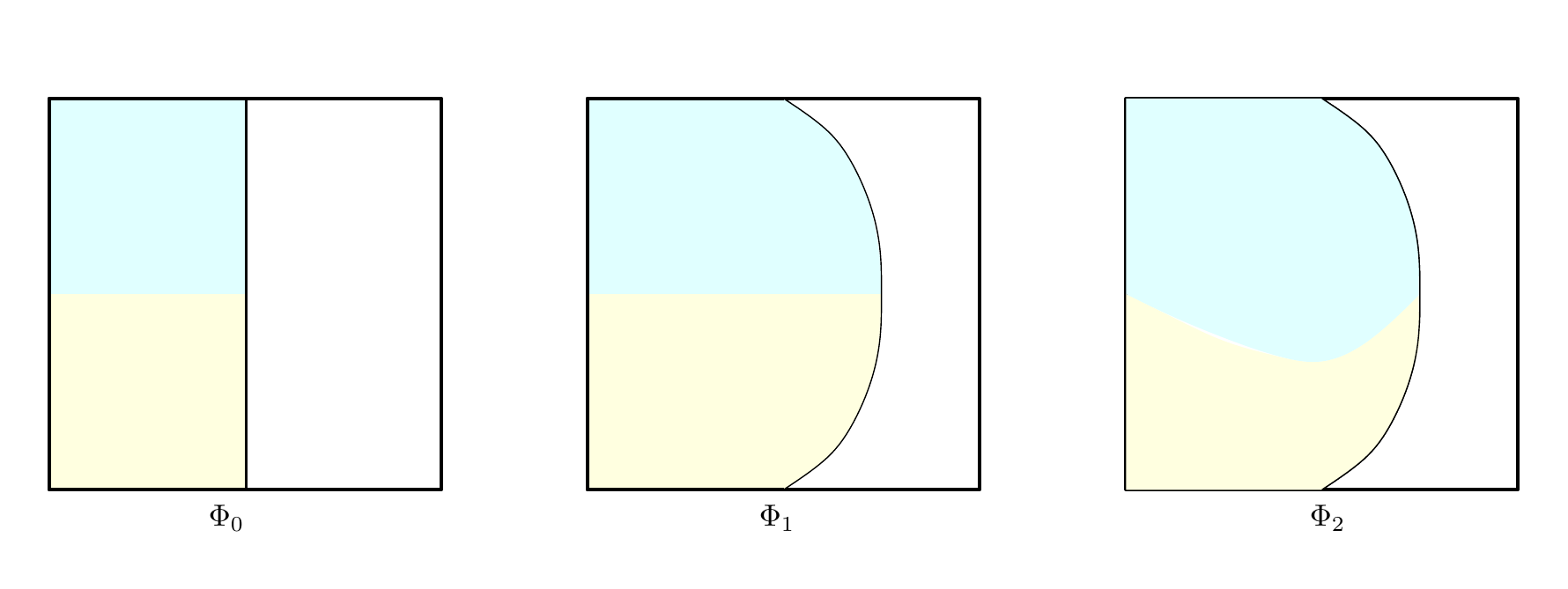}
\caption{$\delta(1,x)=\frac{1}{2}$ and $\delta(2,x)=-0.5$ on $\Lambda_2^1$}
\label{f:phi1phi2}
\end{center}
\end{figure*}

Our primary objective will be to control the derivative of $\Phi_n$ which we can write as
\begin{equation}
\label{e:Phiprimecomposed}
\Phi_{n}'(x)=\prod_{i=1}^{n}\varphi_i'(\varphi_{i+1} \circ \ldots \circ \varphi_n(x)).
\end{equation}
Notice that the product in the above equation is a product of matrices. We will anaylse the following $\cf_n$-measurable approximation to $\Phi_{n}'$,
\begin{equation}
\label{e:yn}
\mathbf{Y}_n=\prod_{i=1}^{n}\E \left[\varphi_i'(\varphi_{i+1} \circ \ldots \circ \varphi_n(X))\mid \cf_n\right].
\end{equation}
We further define the quantities
\begin{align}
\mathbf{W}_{i,n}&=\E \left[\varphi_i'(\varphi_{i+1} \circ \ldots \circ \varphi_n(X))\mid \cf_n\right]-\E \left[\varphi_i'(\varphi_{i+1} \circ \ldots \circ \varphi_{n-1}(X))\mid \cf_{n-1}\right]; \\
\mathbf{V}_{i,n}&=\E \left[\varphi_i'(\varphi_{i+1} \circ \ldots \circ \varphi_{n-1}(X))\mid \cf_n\right]-\E \left[\varphi_i'(\varphi_{i+1} \circ \ldots \circ \varphi_{n-1}(X))\mid \cf_{n-1}\right]; \\
\mathbf{U}_{i,n}&=\E \left[\varphi_i'(\varphi_{i+1} \circ \ldots \circ \varphi_{n}(X))\mid \cf_n\right]-\E \left[\varphi_i'(\varphi_{i+1} \circ \ldots \circ \varphi_{n-1}(X))\mid \cf_{n}\right].
\end{align}
It is clear that $\mathbf{W}_{i,n}=\mathbf{U}_{i,n}+\mathbf{V}_{i,n}$. Observe that
\[
\E \left[\varphi_i'(\varphi_{i+1} \circ \ldots \circ \varphi_n(X))\mid \cf_n\right]-\sum_{j=i+1}^{n} \mathbf{W}_{i,j}=\E[\varphi_i'(X)\mid \cf_i].
\]
Since $\varphi_{i}(X)$ is identity on the boundary of $\Lambda_{i,X}$, Green's Theorem implies that integral of $\varphi_{i}'$ over $\Lambda_{i,X}$ is equal to the integral of $\varphi_i$ over the boundary of $\Lambda_{i,X}$ and hence
\[
\E[\varphi_i'(X)\mid \cf_i]=\mathbf{\rm I}.
\]
It follows that
\begin{equation}
\label{e:ysumprod}
\mathbf{Y}_n=\prod_{i=1}^{n}\left(\mathbf{\rm I}+\sum_{j=i+1}^{n}\mathbf{W}_{i,j}\right).
\end{equation}

\subsection*{Dealing with Twists and Seams:}
While the derivative of $\Psi$ is well behaved in most regions, observe that  $\Psi_{\delta}$ is not differentiable on the \emph{seams} $S$, and the second derivative is unbounded on the \emph{twists} $T$. These shortcomings in smoothness are inherited by the functions $\phi_{i}$. To deal with these issues we shall need the following notations.

For a fixed $x$, let
$$J_{i,x}=\max\{\ell:\Lambda_{i+\ell,x}~\text{intersects}~T_{i,x}\}.$$
Clearly $J_{i,x}$ is almost surely well defined. It measures for how long the $n$-th level dyadic box containing $x$ intersected the \emph{twists} in $\Lambda_{i,x}$.

For a fixed $x$ and $i$, and for $n>i$, let $A_{i,n,x}$ denote the event that $\varphi_{i+1}\circ \varphi_{i+2}\circ \ldots \circ \varphi_{n}(\Lambda_{n-1,x})$ intersects $S_{i,x}$. Let
$$\alpha_{i,n,x}=(n-i)1_{A_{i,n,x}}.$$
By construction the sets $A_{i,n,x}$ is decreasing in $n$. Now let $\alpha_{i,x}=\max\{n:\alpha_{i,n,x}\}$ and define $\beta_{i,x}=J_{i,x}\vee \alpha_{i,x}$. We set $\tilde{\Delta}_{i,x}=|\Delta_{i,x}|2^{\beta_{i,x}/10}$, reweighting $\Delta_{i,x}$ when it is close to a twist or seam.
Finally for $n>i$, we define
$$\tilde{\Delta}_{i,x,n}=|\Delta_{i,x}|2^{(\beta_{i,x}\wedge (n-i))/10}.$$
Notice that $\tilde{\Delta}_{i,x,n}$ is increasing in $n$ and $\tilde{\Delta}_{i,x,n}\leq \tilde{\Delta}_{i,x}$. Also notice that $\tilde{\Delta}_{i,x,n}$ is $\cf_n$ measurable.

\section{Stopping Times}
\label{s:stop}

The primary philosophy of our proof is to keep stretching $A$ on smaller scales but stopping before it violates the Lipschitz constant. To implement this approach we define a series of stopping times. Let $\varepsilon_{1}, \varepsilon_{2}$ be small positive constants that will be chosen later in the proof and set $\varepsilon_{3}=\frac{1}{2}\min\{\gamma, 1-\gamma'\}, \varepsilon_{4}=\frac{\eta}{200}$.
We set
\begin{align}
\label{e:tau1}
\tau_1&:=\tau_1(\varepsilon_1)=\inf\{n:\sum_{i=1}^{n-1}
\tilde{\Delta}_{i,X,n}^{2}>\varepsilon_{1}\},\\
\label{e:tau2}
\tau_2&:=\tau_2(\varepsilon_2)=\inf\{n:\Delta_{n,X}^2>\varepsilon_{2}\},\\
\label{e:tau3}
\tau_3&:=\tau_3(\varepsilon_3)=\inf\{n:|\rho_{n,X}-\lambda(A)|>\varepsilon_{3}\},\\
\label{e:tau4}
\tau_4&:=\tau_4(\varepsilon_4)=\inf\{n:||Y_n(X)-I||_{\infty}>\varepsilon_{4}\}.
\end{align}
Also we define
$$\tau:=\tau_1\wedge \tau_2 \wedge \tau_3 \wedge \tau_4 .$$
It is clear from Observation \ref{o:martingale} that $\tau$ is finite almost surely. We primary work of the next two sections is to prove the following theorem.

\begin{theorem}
\label{t:stoptime}
There exists $\varepsilon_1$ and $\varepsilon_2>0$ such that for the stopping times defined above, we have
$$\P[\tau_4=\tau]< \frac{1}{3}.$$
\end{theorem}

\section{Estimates on $\mathbf{U}$ and $\mathbf{V}$}
In this section we show that for a fixed $i$, on $\{n<\tau\}$, $||\mathbf{U}_{i,n}||$ and $||\mathbf{V}_{i,n}||$ cannot be too large and decays exponentially with $(n-i)$. We start with the estimate on $\mathbf{V}_{i,n}$.
\label{s:estuv}
\begin{proposition}
\label{p:vin}
There exists some absolute constant $C_4>0$ such that for each $i\geq 1$, and $n>i$,  we have
\begin{equation}
\label{e:vinbound}
||\mathbf{V}_{i,n}(X)||1_{\{n<\tau\}}\leq C_4\tilde{\Delta}_{i,X,n}2^{-(n-i)/20}\leq C_4\tilde{\Delta}_{i,X}2^{-(n-i)/20}.
\end{equation}
\end{proposition}

\begin{proof}
Observe that,
\begin{equation}
\label{e:vinbound1}
||\mathbf{V}_{i,n}(X)|| \leq  \max_{x,x'\in \Lambda_{n-1,X}}
||\varphi_i'(\varphi_{i+1} \circ \ldots \circ \varphi_{n-1}(x))-\varphi_i'(\varphi_{i+1} \circ \ldots \circ \varphi_{n-1}(x'))||.
\end{equation}

Observe that, we have on  $\{n<\tau\}$, $\delta(n,X)\leq \frac{2}{\gamma}\Delta_{n,X}$. Set $C_5=\frac{2C_1}{\gamma}$. Now we need to consider two different cases.

{\bf Case 1:} $i+\beta_{i,X}<n$.
In this case, we have that the line joining $x,x'$ does not intersect $S_{i,X}$. Fix a constant $\varepsilon_{5}>0$ sufficiently small. Notice that we have by Proposition \ref{p:lipschitz} that if
$\varepsilon_{2}< \frac{\varepsilon_5}{100(C_5+1)}$, then on $\{n<\tau\}$,
$\varphi_{i+1} \circ \ldots \circ \varphi_{n-1}$ is a bi-Lipschitz continuous function on $\Lambda_{n-1,X}$ with Lipschitz constant at most $(1+\varepsilon_5)^{n-i}$.  Further observe the following. Since $i+\beta_{i,X}<n$, for any point $x\in \Lambda_{n-1,X}$, we have $d(x,T_{i,X})\geq \frac{1}{4}2^{-(i+\beta_{i,X})/2}$. By bi-Lipschitz continuity of $\varphi_{i+1} \circ \ldots \circ \varphi_{n-1}$ it follows that $d(\varphi_{i+1} \circ \ldots \circ \varphi_{n-1}(\Lambda_{n-1,X}),T_{i,X})\geq \frac{1}{4}(1+\varepsilon_{5})^{-(n-i)}2^{-(i+\beta_{i,X})/2}$.

Hence it follows from (\ref{e:vinbound1}) and Proposition \ref{p:psiprimedeviation} that on $\{n \leq \tau\}$ for some absolute constants $C$ and $\varepsilon_{5}$ sufficiently small

\begin{eqnarray}
\label{e:vinbound2}
||\mathbf{V}_{i,n}(X)|| &\leq &  C\Delta_{i,X}\dfrac{\max_{x,x'\in \Lambda_{n-1,X}}|x-x'||\varphi_{i+1} \circ \ldots \circ \varphi_{n-1}||_{lip}}{d(\varphi_{i+1} \circ \ldots \circ \varphi_{n-1}(\Lambda_{n-1,X}),T_{i,X})}\nonumber\\
&\leq & C \dfrac{2^{-n/2}(1+\varepsilon_5)^{n-i}}{(1+10\varepsilon_{4})^{-(n-i)}2^{-(i+\beta_{i,X})/2}}\nonumber\\
&\leq &  C\Delta_{i,X}(1+\varepsilon_5)^{2(n-i)}2^{-(n-i-\beta_{i,X})/2}\nonumber\\
&\leq & C\Delta_{i,X}2^{\beta_{i,X}/10}2^{-(n-i)/20}
\end{eqnarray}
where the final inequality follows by taking $\varepsilon_5$ sufficiently small.

{\bf Case 2:} $i+\beta_{i,X}\geq n$.

In this case it follows from Proposition \ref{p:psiprimedeviation} that we have on $\{i<\tau\}$
$$\max_{x,x'\in \Lambda_{i,X}}
||\varphi_i'(x)-\varphi_i'(x')||\leq C\Delta_{i,X}.$$
It follows now from (\ref{e:vinbound1}) that on $\{n<\tau\}$,

\begin{equation}
\label{e:vinbound3}
||\mathbf{V}_{i,n}(X)||\leq C\Delta_{i,X}\leq C\Delta_{i,X}2^{(n-i)/10}2^{-(n-i)/20}.
\end{equation}

The proposition now follows from (\ref{e:vinbound2}) and (\ref{e:vinbound3}) by choosing $C_4$ appropriately.
\end{proof}

We have a similar result for $\mathbf{U}_{i,n}$ where we get a $\tilde{\Delta}_{i,X}\Delta_{n,X}$ term instead of the $\Delta_{i,X}$ term in the above proposition. 

\begin{proposition}
\label{p:uin}
There exists some absolute constant $C_6>0$ such that for each $i\geq 1$, and $n>i$,  we have
\begin{equation}
\label{e:uinbound}
||\mathbf{U}_{i,n}(X)||1_{\{n<\tau\}}\leq C_6\tilde{\Delta}_{i,X,n}\Delta_{n,X}2^{-(n-i)/20}.
\end{equation}
\end{proposition}

To prove Proposition \ref{p:uin} we need some additional work to deal with the possibility that $\varphi_{i+1}\circ \cdots \varphi_{n-1}(\Lambda_{n-1,X})$ might intersect $S_{i,X}$. To this end we make the following definition.  For $x\in \Lambda_{n,X}$, let $A_{i,n,x}$ denote the event that the line segment joining $\varphi_{i+1} \circ \ldots \circ \varphi_{n-1}(x)$ and $\varphi_{i+1} \circ \ldots \circ \varphi_{n}(x)$ intersects $S_{i,X}$. Let
$$\mathcal{A}_{i,n,X}=\{x\in\Lambda_{n,X}: 1_{\{A_{i,n,x}\}}>0\}.$$

We have the following lemma bounding the measure of the set $\mathcal{A}_{i,n,X}$.
\begin{lemma}
\label{l:claim2}
For some absolute constant $C>0$, we have on $\{n<\tau\}$

\begin{equation}
\label{e:uinbound5}
\lambda(\mathcal{A}_{i,n,X})\leq C\Delta_{n,X}2^{(n-i)/20}2^{-n}.
\end{equation}
\end{lemma}

\begin{proof}
Denote the two seams in $\Lambda_{i,X}$ by $S^1_{i,X}$ and $S^2_{i,X}$ respectively. For $x\in \Lambda_{n,X}$, let $A^1_{i,n,x}$ denote the event that the line segment joining $\varphi_{i+1} \circ \ldots \circ \varphi_{n-1}(x)$ and $\varphi_{i+1} \circ \ldots \circ \varphi_{n}(x)$ intersects $S^1_{i,X}$. Let

$$\mathcal{A}^1_{i,n,X}=\{x\in\Lambda_{n,X}: 1_{\{A_{i,n,x}\}}>0\}.$$

By symmetry, it suffices to prove that for some absolute constant $C>0$, we have on $\{n<\tau\}$

\begin{equation}
\label{e:uinbound6}
\lambda(\mathcal{A}^1_{i,n,X})\leq C\Delta_{n,X}2^{(n-i)/20}2^{-n}.
\end{equation}

Interpreting $S^1_{i,X}$ as a directed simple curve there exist a \emph{first point} $y\in S^1_{i,X}$ where $S^1_{i,X}$ enters $\varphi_{i+1} \circ \ldots \circ \varphi_{n-1}(\Lambda_{n,X})$ and a \emph{last point} $y'\in S^1_{i,X}$ where $S^1_{i,X}$ exits $\varphi_{i+1} \circ \ldots \circ \varphi_{n-1}(\Lambda_{n,X})$. If $\varepsilon_5$ is such that the bi-Lipschitz constant of $\varphi_{i+1} \circ \ldots \circ \varphi_{n-1}$ is at most $(1+\varepsilon_5)^{n-i}$ then we get that $|y-y'|\leq C(1+\varepsilon_5)^{(n-i)}2^{-n/2}$ for some absolute constant $C>0$. Let $S^1_{i,X}(y,y')$ denote the curve segment $S^1_{i,X}$ from $y$ to $y'$. Let $\ell(y,y')$ denote the length $S^1_{i,X}$ from $y$ to $y'$. It then follows that $\ell(y,y')\leq C(1+\varepsilon_5)^{(n-i)}2^{-n/2}$ for some absolute constant $C>0$. Now define

$$A_{y,y',C'}=\{x\in \Lambda_{i,X}:\exists z\in S^1_{i,X}(y,y') ~\text{such that} |x-z|\leq C'\Delta_{n,X}(1+\varepsilon_5)^{(n-i)}2^{-n/2}\}.$$

Clearly $x\in \mathcal{A}^1_{i,n,X}$ implies $\varphi_{i+1} \circ \ldots \circ \varphi_{n-1}(x)\in A_{y,y',C'}$ if $|\varphi_n(x)-x|\leq C'\Delta_{n,X}2^{-n/2}$. It follows that

$$\lambda(\mathcal{A}^1_{i,n,X})\leq (1+\varepsilon_5)^{(n-i)}\lambda(A_{y,y',C'}).$$

Clearly for some constant $C>0$,
$$\lambda(A_{y,y',C'})\leq C\ell(y,y')\Delta_{n,X}(1+\varepsilon_5)^{(n-i)}2^{-n/2}\leq C\Delta_{n,X}(1+\varepsilon_5)^{2(n-i)}2^{-n}$$

and hence we have

$$\lambda(\mathcal{A}^1_{i,n,X})\leq C\Delta_{n,X}(1+\varepsilon_5)^{3(n-i)}2^{-n}.$$

By taking $\varepsilon_5$ sufficiently small we establish (\ref{e:uinbound6}) and the lemma is proved.
\end{proof}

We also need the following lemma.

\begin{lemma}
\label{l:claim1}
For some absolute constant $C>0$, we have for each $x\in \Lambda_{n,x}$, on $\{n<\tau\}$

\begin{equation}
\label{e:uinbound4}
\dfrac{|\varphi_{i+1} \circ \ldots \circ \varphi_{n}(x)-\varphi_{i+1} \circ \ldots \circ \varphi_{n-1}(x)|}{d(\varphi_{i+1} \circ \ldots \circ \varphi_{n}(x),T_{i,X})}\leq C\Delta_{n,X}2^{(n-i)/20}.
\end{equation}
\end{lemma}

\begin{proof}
It follows from Proposition \ref{p:distance} that for some absolute constant $C$ for all $x\in \Lambda_{n,X}$ we have $|\varphi_n(x)-x|\leq C\Delta_{n,X}d(x,T_{i,X})$. Now for $\varepsilon_{5}$ as in the previous case, we have $$|\varphi_{i+1} \circ \ldots \circ \varphi_{n}(x)-\varphi_{i+1} \circ \ldots \circ \varphi_{n-1}(x)|\leq (1+\varepsilon_5)^{(n-i)}d(x,T_{i,X})$$ and also
$$d(\varphi_{i+1} \circ \ldots \circ \varphi_{n}(x),T_{i,X})\geq (1+\varepsilon_{5})^{-(n-i)}d(x,T_{i,X}).$$
Taking $\varepsilon_{5}$ sufficiently small, (\ref{e:uinbound4}) follows from the above two equations.
\end{proof}

Now we are ready to prove Proposition \ref{p:uin}.

\begin{proof}[Proof of Proposition \ref{p:uin}]
For the proof of this proposition also we need to consider two cases.

{\bf Case 1:} $i+\beta_{i,X}<n$.

In this case, we have that $\varphi_{i+1}\circ \cdots \varphi_{n-1}(\Lambda_{n-1,X})$ does not intersect $S_{i,X}$. Hence it follows that by arguments similar to those in the proof of Propositiion \ref{p:vin} that on $\{n<\tau\}$ we have using Proposition \ref{p:psiprimedeviation}

\begin{equation}
\label{e:uinbound1}
||\mathbf{U}_{i,n}(X)|| \leq  \frac{2C_3}{\gamma}\Delta_{i,X} \max_{x\in \Lambda_{n,X}}\dfrac{|\varphi_{i+1} \circ \ldots \circ \varphi_{n}(x)-\varphi_{i+1} \circ \ldots \circ \varphi_{n-1}(x)|}{d(\varphi_{i+1} \circ \ldots \circ \varphi_{n}(x),T_{i,X})}.
\end{equation}

Now observe that since $i+\beta_{i,X}<n$, we have that $d(\Lambda_{n,X},T_{i,X})\geq \frac{1}{4}2^{-(i+\beta_{i,X})/2}$. Choosing $\varepsilon_5$ and $\varepsilon_2$ as in the proof of Proposition \ref{p:vin} such that the bi-Lipschitz constant of $\varphi_{i+1} \circ \ldots \circ \varphi_{n}$ is at most $(1+\varepsilon_5)^{n-i}$ we get that $d(\varphi_{i+1} \circ \ldots \circ \varphi_{n}(\Lambda_{n,X}),T_{i,X})\geq \frac{1}{4}(1+\varepsilon_5)^{-(n-i)}2^{-(i+\beta_{i,X})/2}$. Also observe that it follows from Proposition \ref{p:distance} that for some absolute constant $C>0$ we have

$$\max_{x\in \Lambda_{n,X}}|\varphi_{n}(x)-x|\leq  C\Delta_{n,X}2^{-n/2}$$

It follows now from (\ref{e:uinbound1}) that on $\{n<\tau\}$ for some absolute constant $C$
\begin{eqnarray}
\label{e:uinbound2}
||\mathbf{U}_{i,n}(X)|| &\leq & C\Delta_{i,X} \dfrac{||\varphi_{i+1} \circ \ldots \circ \varphi_{n-1}||_{lip} C\Delta_{n,X}}{(1+\varepsilon_5)^{-(n-i)}2^{-(i+\beta_{i,X})/2}}\nonumber\\
&\leq & C\Delta_{i,X}\Delta_{n,X}(1+\varepsilon_5)^{2(n-i)}2^{-(n-i-\beta_{i,X})/2}\nonumber\\
&\leq & C\Delta_{i,X}\Delta_{n,X}2^{\beta_{i,X}/10}2^{-(n-i)/20}.
\end{eqnarray}

{\bf Case 2:} $i+\beta_{i,X}\geq n$.
It follows from Proposition \ref{p:psiprimedeviation} that

\begin{equation}
\label{e:uinbound3}
||\mathbf{U}_{i,n}(X)|| \leq  C\Delta_{i,X}\left(2^{n}\lambda(\mathcal{A}_{i,n,X})+\max_{x\in \Lambda_{n,X}\setminus \mathcal{A}_{i,n,X}}\dfrac{|\varphi_{i+1} \circ \ldots \circ \varphi_{n}(x)-\varphi_{i+1} \circ \ldots \circ \varphi_{n-1}(x)|}{d(\varphi_{i+1} \circ \ldots \circ \varphi_{n}(x),T_{i,X})}\right).
\end{equation}

Using (\ref{e:uinbound3}), Lemma \ref{l:claim2} and Lemma \ref{l:claim1} we get for some absolute constant $C>0$
$$||\mathbf{U}_{i,n}(X)||\leq C\Delta_{i,X}\Delta_{n,X}2^{(n-i)/20} \leq C\Delta_{i,X}\Delta_{n,X}2^{(n-i)/10}2^{-(n-i)/20}\leq C\tilde{\Delta}_{i,X,n}\Delta_{n,X}2^{-(n-i)/20}.$$

Proof of the proposition is completed by choosing $C_6$ appropriately.
\end{proof}

\section{Bounding $\mathbf{Y}_{n}$}
\label{s:esty}
Our next step is to prove Theorem \ref{t:stoptime}. That is, we need to show that it is unlikely that $||\mathbf{Y}_n-{\rm{\mathbf{I}}}||$ becomes large before either  $\rho_{n,X}$ deviates significantly from $\lambda(A)$, $\Delta_{n,X}$ becomes sufficiently large or $\sum_{k\leq n} \tilde{\Delta}_{k,X,n}^2$ becomes sufficiently large. For this purpose we construct a matrix-valued martingale $\mathbf{M}_n$.

\subsection{Constructing $\mathbf{M}_n$}
Define a sequence $\{\mathbf{M}_n\}_{n\geq 1}$ of matrix-valued random objects as follows.

\begin{enumerate}
\item Set $\mathbf{M}_1={\rm \mathbf{I}}$.
\item For $n\geq 1$, set
\begin{equation}
\label{e:mndifference}
\mathbf{M}_{n+1}-\mathbf{M}_{n}=\sum_{k=1}^{n}\prod_{i=1}^{k}\left({\rm \mathbf{I}}+\sum_{j=i+1}^{n}\mathbf{W}_{i,j}\right)\mathbf{V}_{k,n+1}\prod_{i=k+1}^{n}\left({\rm\mathbf{I}}+
\sum_{j=i+1}^{n}\mathbf{W}_{i,j}\right).
\end{equation}
\end{enumerate}

Clearly, it follows that $\mathbf{M}_n$ is $\cf_n$ measurable and since $\E[\mathbf{V}_{k,n+1}\mid \cf_n]=0$ it follows that $\mathbf{M}_n$ is a martingale with respect to the filtration $\{\cf_n\}$.

Let $\tau_1$, $\tau_2$, $\tau_3$ be defined by (\ref{e:tau1}), (\ref{e:tau2}) and (\ref{e:tau3}) respectively. Define the stopping times $\tau_5$ and $\tau_6$ by

\begin{equation}
\label{e:tau5}
\tau_5=\{\inf n: ||Y_n||\geq 2\}.
\end{equation}

\begin{equation}
\label{e:tau6}
\tau_6=\tau_6(\varepsilon_6)=\{\inf n: ||\mathbf{M}_n-{\rm \mathbf{I}}||\geq \varepsilon_6 \}.
\end{equation}

Let us define

$$\tau'= \tau_1\wedge \tau_2\wedge \tau_3\wedge \tau_5$$

Observe that $\tau\leq \tau'$.

We shall prove the following theorem.

\begin{theorem}
\label{t:mn}
Set $\varepsilon_6=\eta/800$. Then there exists $\varepsilon_1$, $\varepsilon_2>0$ such that we have
\begin{equation}
\label{e:mntheorem}
\P[\tau_6< \tau']\leq \frac{1}{5}.
\end{equation}
\end{theorem}

Before proving Theorem \ref{t:mn}, observe the following. It follows from (\ref{e:ysumprod}) that

\begin{equation}
\label{e:mndiff}
\mathbf{M}_{n+1}-\mathbf{M}_{n}=\sum_{k=1}^{n}\mathbf{Y}_n\left(\prod_{i=k+1}^{k}\left({\rm\mathbf{I}}
+\sum_{j=i+1}^{n}\mathbf{W}_{i,j}\right)\right)^{-1}\mathbf{V}_{k,n+1}\prod_{i=k+1}^{n}\left({\rm \mathbf{I}}+\sum_{j=i+1}^{n}\mathbf{W}_{i,j}\right).
\end{equation}

Choose $\varepsilon_5$ sufficiently small and set $\varepsilon_2=\frac{\varepsilon_5}{100(C_5+1)}$. By choosing $\varepsilon_5$ sufficiently small we have for each $i\in [k+1, n]$ we have $||{\rm \mathbf{I}}+\sum_{j=i+1}^{n}\mathbf{W}_{i,j}||\leq (1+\varepsilon_5)$, $||\left({\rm \mathbf{I}}+\sum_{j=i+1}^{n}\mathbf{W}_{i,j}\right)^{-1}||\leq (1+\varepsilon_5)$ using Proposition \ref{p:psideltaderivative}.
It follows now using Proposition \ref{p:vin} that on $\{n< \tau'\}$ we have for $\varepsilon_5$ small enough and for some absolute constant $C>0$
\begin{eqnarray}
\label{e:mndiffbound1}
||\mathbf{M}_{n+1}-\mathbf{M}_{n}||^2 & \leq & C||\mathbf{Y}_n||^2\left(\sum_{k=1}^{n} 2^{-(n+1-k)/20}(1+\varepsilon_5)^{2(n+1-k)}\tilde{\Delta}_{k,X,n}\right)^2\nonumber\\
&\leq & C||\mathbf{Y}_n||^2 \left(\sum_{k=1}^{n} 2^{-(n+1-k)/50}\tilde{\Delta}_{k,X,n}\right)^2\nonumber\\
&\leq & C||\mathbf{Y}_n||^2 \left(\sum_{k=1}^{n} 2^{-(n+1-k)/100}\right)\left(\sum_{k=1}^{n} 2^{-(n+1-k)/100}\tilde{\Delta}_{k,X,n}^2\right)\nonumber\\
&\leq & C\sum_{k=1}^{n} 2^{-(n+1-k)/100}\tilde{\Delta}_{k,X,n}^2.
\end{eqnarray}
where the third inequality is by the Cauchy-Schwartz Inequality.
Hence on the event $\{n+1<\tau'\}$
we have
\begin{eqnarray}
\label{e:mndiffbound2}
\sum_{i=1}^{n} \E[||\mathbf{M}_{i+1}-\mathbf{M}_{i}||^2\mid \cf_i] &\leq & C\sum_{i=1}^{n}\sum_{k=1}^{i}2^{-(i+1-k)/100}\tilde{\Delta}_{k,X}^2\nonumber\\
&\leq & C\sum_{i=1}^{n}\tilde{\Delta}_{i,X,n}^2 \leq C_{8}\varepsilon_{1}.
\end{eqnarray}
Similarly to (\ref{e:mndiffbound1}), we have that on the event $\{n <\tau'\}$ for some absolute constant $C>0$
\begin{eqnarray}
\label{e:mnabsdifference}
||\mathbf{M}_{n+1}-\mathbf{M}_{n}|| &\leq &  C||Y_n||\sum_{k=1}^{n}2^{-(n+1-k)/50}\tilde{\Delta}_{k,X,n}\nonumber\\
&\leq & C\sum_{k=1}^{n}2^{-(n+1-k)/50}\left(2^{(n-k)/10}\sqrt{\varepsilon_{2}}\wedge \sqrt{\varepsilon_1}\right)\nonumber\\
&\leq & 3\varepsilon_{1}^{1/2}
\end{eqnarray}
for $\varepsilon_{2}$ sufficiently small.

Having bounded the quadratic variation and increment size of $\mathbf{M}_{n+1}$ we use the following inequality for tail probability of a martingale, which is a generalisation of Bernstein's inequality.

\begin{theorem}[Freedman, 1975 \cite{Freedman75}]
\label{t:fried}
Let $\{X_n\}_{n\geq 1}$ be a martingale with respect to the filtration $\{\cf_n\}$, with $|X_{k+1}-X_{k}|\leq R$ almost surely. Let $Y_{n}=\sum_{i=1}^{n-1}\E\left[(X_{i+1}-X_{i})^2\mid \cf_i\right]$. Then for each $t$
\begin{equation}
\label{e:fried}
\P[\exists n: X_n-X_1> t, Y_n\leq \sigma ^2]\leq \exp \left\{-\frac{t^2}{2(\sigma^2 + Rt/3)} \right\}.
\end{equation}
\end{theorem}

Now we are ready to prove Theorem \ref{t:mn}.

\begin{proof}[Proof of Theorem \ref{t:mn}]
Choose $\varepsilon_1$ sufficiently small so that $\varepsilon_6\geq \left(\sqrt{20C_8}\vee \sqrt{20}\right)\varepsilon_1^{1/2}$. For $i,j=1,2$ let the $(i,j)$-th entry of $\mathbf{M}_{n}$ be $M_n^{i,j}$. Consider the martingales $\{X_n^{i,j}\}=\{M_{n\wedge \tau'}^{i,j}\}$. It follows from Theorem \ref{t:fried} using (\ref{e:mndiffbound2}) and (\ref{e:mnabsdifference}) that
$$\P[\exists n<\tau': |X_n^{i,j}-\delta_{i,j}|\geq \varepsilon_6]\leq 2e^{-5}.$$
Taking a union bound over different values of $i$ and $j$ it follows that
$$\P[\exists n<\tau': ||\mathbf{M}_n-{\rm \mathbf{I}}||\geq \varepsilon_6]\leq 8e^{-5}< \frac{1}{5}.$$
This finishes the proof of the theorem.
\end{proof}

\subsection{Bounding $\mathbf{Y}_n-\mathbf{M}_n$}

We define $\mathbf{D}_n=\mathbf{Y}_n-\mathbf{M}_n$. Note that $\mathbf{D}_1=\mathbf{0}$. We have the following lemma.

\begin{lemma}
\label{l:dn}
Set $\varepsilon_8=\eta/800$. There exists $\varepsilon_1,\varepsilon_2>0$, satisfying the conclusion of Theorem \ref{t:mn} such that we have
\begin{equation}
\label{e:dnbound}
||\mathbf{D}_{t}||1_{\{t<\tau'\}} \leq \varepsilon_8.
\end{equation}
\end{lemma}

\begin{proof}
Observe that on $\{t<\tau'\}$, we have
$$||\mathbf{D}_t||\leq \sum_{n=1}^{t-1}||\mathbf{D}_{n+1}-\mathbf{D}_{n}||.$$
Expanding out $D_n$ we have
\begin{eqnarray}
\label{e:dndiffbound1}
\mathbf{D}_{n+1}-\mathbf{D}_{n}&=& \sum_{k=1}^{n}\left(\prod_{i=1}^{k}\left({\rm\mathbf{I}}+\sum_{j=i+1}^{n}\mathbf{W}_{i,j}\right)\mathbf{U}_{k,n+1}\prod_{i=k+1}^{n}\left({\rm\mathbf{I}}
+\sum_{j=i+1}^{n}\mathbf{W}_{i,j}\right)\right)\nonumber\\
&+ & \sum_{S\subseteq [n],|S|\geq 2}\left(\prod_{i=1}^{n}\left(1_{\{i\in S\}}\mathbf{W}_{i,n+1}+ 1_{\{i\notin S\}}\left({\rm\mathbf{I}}+\sum_{j=i+1}^{n}\mathbf{W}_{i,j}\right)\right)\right).
\end{eqnarray}

Call the first term on the right hand side of the above equation $\mathbf{A}_n$, call the second term $\mathbf{B}_n$. Choosing $\varepsilon_5=100(C_5+1)\varepsilon_2$ and arguing as in (\ref{e:mndiffbound1}) we get that for some absolute constant $C>0$ we get that on $\{n+1<\tau'\}$,

\begin{eqnarray}
\label{e:anbound}
||\mathbf{A}_n|| & \leq & C\sum_{k=1}^{n} 2^{-(n+1-k)/20}(1+\varepsilon_5)^{2(n+1-k)}\tilde{\Delta}_{k,X,n+1}\Delta_{n+1,X}\nonumber\\
&\leq & C \sum_{k=1}^{n} 2^{-(n+1-k)/50}\left(\tilde{\Delta}_{k,X,n+1}^2+\Delta_{n+1,X}^2\right)\nonumber\\
&\leq & C\Delta_{n+1,X}^2+ C\sum_{k=1}^{n} 2^{-(n+1-k)/50}\tilde{\Delta}_{k,X,n+1}^2.
\end{eqnarray}
It follows that on $\{t<\tau'\}$ for some absolute constants $C,C_9>0$
\begin{eqnarray}
\label{e:anbound1}
\sum_{n=1}^{t-1}||\mathbf{A}_n|| &\leq & C\sum_{n=1}^{t-1}\Delta_{n+1,X}^2+C\sum_{n=1}^{t-1}\sum_{k=1}^{n} 2^{-(n+1-k)/50}\tilde{\Delta}_{k,X,t}^2\nonumber\\
&\leq & C\sum_{i=1}^{t}\tilde{\Delta}_{i,x,t}^2 \leq C_9(\varepsilon_1+\varepsilon_2).
\end{eqnarray}
For obtaining a bound on $\mathbf{B}_n$, observe the following. Fix $S\subseteq [n]$ with $|S|\geq 2$. It follows from Proposition \ref{p:vin} and Proposition \ref{p:uin} that for some absolute constant $C>0$ we have $|\mathbf{W}_{i,n+1}|\leq C\tilde{\Delta}_{i,X}2^{-(n+1-i)/20}$ for each $i\in S$. Arguing similarly as in (\ref{e:mndiffbound1}) it follows by taking $\varepsilon_2$ sufficiently small, on $\{n+1<\tau'\}$, we have
\begin{eqnarray}
\label{e:bnbound1}
\prod_{i=1}^{n}\left(1_{\{i\in S\}}\mathbf{W}_{i,n+1}+ 1_{\{i\notin S\}}\left({\rm \mathbf{I}}+\sum_{j=i+1}^{n}\mathbf{W}_{i,j}\right)\right) & \leq & ||Y_n||\prod_{i\in S} C\tilde{\Delta}_{i,X,n+1}(1+\varepsilon_5)^{2(n+1-i)}2^{-(n+1-i)/20}\nonumber\\
&\leq & \prod_{i\in S} C\tilde{\Delta}_{i,X,n+1}2^{-(n+1-i)/50}.
\end{eqnarray}
Summing over all $S\subseteq [n]$ with $|S|\geq 2$ we get that
\begin{equation}
\label{e:bnbound2}
||\mathbf{B}_n||\leq \prod_{i=1}^{n}\left(1+ C\tilde{\Delta}_{i,X,n+1}2^{-(n+1-i)/50}\right)-1-\sum_{i=1}^{n} C\tilde{\Delta}_{i,X,n+1}2^{-(n+1-i)/50}.
\end{equation}

Now observe that on $\{n+1<\tau'\}$, by choosing $\varepsilon_2$ sufficiently small we have
$$\sum_{i=1}^{n} C\tilde{\Delta}_{i,X,n+1}2^{-(n+1-i)/50}\leq \frac{1}{10}$$
by the argument used in (\ref{e:mnabsdifference}). It then follows that
\begin{equation}
\label{e:bnbound3}
||\mathbf{B}_n||\leq 2C\sum_{i=1}^{n} \tilde{\Delta}_{i,X,n+1}2^{-(n+1-i)/50}\leq 10C\sum_{i=1}^{n} \tilde{\Delta}_{i,X,n+1}^22^{-(n+1-i)/200},
\end{equation}
where the final step follows from the Cauchy-Schwarz inequality.

Summing over $n$ we get on $\{t<\tau'\}$, for some absolute constant $C_{10}>0$

\begin{equation}
\label{e:bnbound4}
\sum_{n=1}^{t-1}||\mathbf{B}_n||\leq C_{10}\sum_{i=1}^{t} \tilde{\Delta}_{i,X,t}^2\leq C_{10}(\varepsilon_1+\varepsilon_2) ,
\end{equation}

It now follows from (\ref{e:dndiffbound1}), (\ref{e:anbound1}) and (\ref{e:bnbound4}) that we have

$$||\mathbf{D}_{t}||1_{\{t<\tau'\}} \leq (C_9+C_{10})(\varepsilon_1+\varepsilon_2).$$

Choosing $\varepsilon_1$ and $\varepsilon_2$ sufficiently small such that $(C_9+C_{10})(\varepsilon_1+\varepsilon_2)\leq \varepsilon_8$, we complete the proof of the lemma.
\end{proof}

\subsection{Proof of Theorem \ref{t:stoptime}}

Now we are ready to prove Theorem \ref{t:stoptime}.

\begin{proof}
Observe that we have $\varepsilon_6+\varepsilon_8<\varepsilon_4$. Fix $\varepsilon_1$, $\varepsilon_2>0$  such that Theorem \ref{t:mn} and Lemma \ref{l:dn} holds. It follows from and Lemma \ref{l:dn}, that on $\{\tau_6\geq \tau'\}$, for all $n<\tau'$ we have
$$||\mathbf{Y}_n-{\rm\mathbf{I}}||\leq  \varepsilon_6+\varepsilon_8< \varepsilon_4.$$

Proof of Theorem \ref{t:stoptime} is then completed using Theorem \ref{t:mn}.
\end{proof}

\section{Proof of Theorem \ref{t:stretchingfixed}}
\label{s:final}
We complete the proof of Theorem \ref{t:stretchingfixed} in this section. We first need the following lemmas.

\begin{lemma}
\label{l:phitaulip}
Set $\varepsilon_9=\eta/400$. Then we can choose $\varepsilon_1$, $\varepsilon_2>0$, in Theorem \ref{t:stoptime} so that we have on $\{n<\tau\}$, $||\Phi'_{n}-\mathbf{Y}_n||\leq \varepsilon_9$.
\end{lemma}

\begin{proof}
Write
$$\mathbf{\xi}_{i,n}=\varphi_i'(\varphi_{i+1} \circ \ldots \circ \varphi_n)-\E[\varphi_i'(\varphi_{i+1} \circ \ldots \circ \varphi_n)\mid \cf_n].$$
Observe that $||\mathbf{\xi}_{i,n}(X)||$ is upper bounded by the right hand side of (\ref{e:vinbound1}) with $n$ there replaced by $(n+1)$. Now arguing as in the proof of Proposition \ref{p:vin} it follows that for some absolute constant $C>0$ we have
\begin{equation}
\label{e:xiin}
||\mathbf{\xi}_{i,n}(X)||\leq C\tilde{\Delta}_{i,X,n}2^{-(n-i)/20}.
\end{equation}
Notice that
\begin{equation}
\label{e:phiprime}
\Phi'_{n}= \prod_{i=1}^n \left({\rm \mathbf{I}}+ \sum_{j=i+1}^{n}\mathbf{W}_{i,j} +\mathbf{\xi}_{i,n}\right).
\end{equation}
It now follows that
\begin{equation}
\label{e:phiprime1}
\Phi'_{n}-\mathbf{Y}_n=\sum_{\emptyset\neq S\subseteq [n]} \prod_{i=1}^n \left(1_{i\notin S}\left({\rm \mathbf{I}}+ \sum_{j=i+1}^{n}\mathbf{W}_{i,j}\right) +1_{i\in S}\mathbf{\xi}_{i,n}\right).
\end{equation}
An argument similar to the one used in (\ref{e:bnbound2}) gives that for $\varepsilon_2$ sufficiently small, we have on $\{n<\tau\}$
\begin{equation}
\label{e:phiprime2}
||\Phi'_{n}-\mathbf{Y}_n||\leq  \prod_{i=1}^n \left(1+C\tilde{\Delta}_{i,x,n}2^{-(n-i)/100}\right)-1\leq 3C\sum_{i=1}^{n}\tilde{\Delta}_{i,x,n}2^{-(n-i)/100}.
\end{equation}
Using the Cauchy-Schwarz inequality as in (\ref{e:bnbound3}) we get that for some absolute constat $C>0$ we have
$$||\Phi'_{n}-\mathbf{Y}_n||\leq C(\varepsilon_1+\varepsilon_2).$$
The lemma follows by taking $\varepsilon_1$ and $\varepsilon_2$ sufficiently small.
\end{proof}

\begin{lemma}
\label{l:deltatildedelta}
We can choose $\varepsilon_1$, $\varepsilon_2>0$ in Theorem \ref{t:stoptime} such that for some absolute constant $C_{11}>0$ we have
\begin{equation}
\label{e:deltadelta}
\E[\tilde{\Delta}_{i,X,\tau}^2\mid \cf_{i+1}]\leq C_{11}\Delta_{i,X}^2.
\end{equation}
\end{lemma}

\begin{proof}
The above lemma is an immediate consequence of the following lemma.
\end{proof}

\begin{lemma}
We can choose $\varepsilon_1$, $\varepsilon_2>0$ in Theorem \ref{t:stoptime} such that we have $\P[\beta_{i,X}\geq (n-i), \tau \geq n\mid \cf_{i+1}]\leq 10\times 2^{-(n-i)/3}$, for $n$ sufficiently large.
\end{lemma}

\begin{proof}
It is clear from definition that $\P[J_{i,X}\geq (n-i)\mid \cf_{i+1}] \leq 2\times 2^{-(n-i)/3}$. Hence it suffices to show that $\P[\alpha_{i,X}\geq (n-i), \tau>n\mid \cf_{i+1}]\leq 8\times 2^{-(n-i)/3}$. It is easy to see that it suffices to show $\P[\alpha_{i,X}\geq (n-i), \tau\geq n\mid \cf_{i}]\leq 4\times 2^{-(n-i)/3}$.

From definition, if $\{\alpha_{i,X}\geq (n-i),n\leq \tau\}$, then $\varphi_{i+1}\circ \cdots \varphi_{n}(\Lambda_{n,X})=\varphi_{i+1}\circ \cdots \varphi_{n-1}(\Lambda_{n,X})$ intersects $S_{i,X}$. Notice that the total length of the curve(s) $S_{i,X}$ is at most $C_{12}2^{-i/2}$ for some absolute constant $C_{12}$. Let $\varepsilon_5$ be a constant such that on $\{n<\tau\}$, $\varphi_{i+1}\circ \cdots \varphi_{n}$ is bi-Lipschitz with Lipschitz constant at most $(1+\varepsilon_5)^{n-i}$. It follows that $\varphi_{i+1}\circ \cdots \varphi_{n\wedge \tau -1}$ is also bi-Lipschitz with Lipschitz constant at most $(1+\varepsilon_5)^{n-i}$. Note that, as before, $\varepsilon_5$ can be made arbitrarily small by taking $\varepsilon_2$ small. Hence it follows that there exists a set $\mathcal{M}$ of $N=8C_{12}(1+\varepsilon_5)^{n-i}2^{(n-i)/2}$ points on $S_{i,X}$ such that any point on $S_{i,X}$ is at most distance $\frac{1}{4}(1+\varepsilon_5)^{(n-i)}2^{-n/2}$ from some point in $\mathcal{M}$. Let $\mathcal{M}=\{x_1,x_2,\ldots, x_{N}\}$. It follows that any point on $(\varphi_{i+1}\circ\cdots \circ \varphi_{n\wedge \tau-1})^{-1}(S_{i,X})$ is at most at distance $\frac{1}{4}2^{-n/2}$ from $(\varphi_{i+1}\circ\cdots \circ \varphi_{n\wedge \tau-1})^{-1}(x_k)$ for some $k$. It follows that for
$$\P[\Lambda_{n,X} ~\text{intersects}~ \varphi_{i+1}\circ\cdots \circ \varphi_{n-1})^{-1}(S_{i,X}),\tau\geq n\mid \cf_i]\leq 32C_{12}(1+\varepsilon_5)^{n-i}2^{-(n-i)/2}\leq 2^{-(n-i)/3}$$
by taking $\varepsilon_5$ sufficiently small completing the lemma.
\end{proof}

\begin{lemma}
\label{l:phitau}
We have for the stopping time $\tau$ w.r.t. the filtration $\cf_i$,
\begin{equation}
\label{e:phitau}
\lambda(\Phi_{\tau-1}(A))=\frac{\E \rho_{\tau,X}^2}{\E \rho_{\tau,X}}=\lambda(A)+\frac{1}{\lambda(A)}{\rm Var}(\rho_{\tau,X}).
\end{equation}
\end{lemma}

\begin{proof}
Let $W_1,W_2,\ldots $ be the disjoint (except may be at the boundary) $\tau$ level dyadic boxes (i.e., $W_{k}=\Lambda_{\tau,X}$ on $\{X\in W_k\}$) such that $\cup_{k} W_k=\Lambda_1$. It follows from the definition od $\varphi_{\tau-1}$ that for each $k$ we have that subsets of $W_k$ are expanded uniformly
\begin{equation}
\label{e:phitau1}
\frac{\lambda(\Phi_{\tau-1}(A\cap W_k))}{\lambda(\Phi_{\tau-1}(W_k))}=\frac{\lambda(A\cap W_k)}{\lambda(W_k)},
\end{equation}
and that the measure of the image of a box is proportion to its density
\begin{equation}
\label{e:phitau2}
\frac{\lambda(\Phi_{\tau-1}(W_k))}{\lambda(W_k)}=\frac{\lambda(A\cap W_k)}{\lambda(W_k)\lambda(A)}.
\end{equation}
Combining (\ref{e:phitau1}) and (\ref{e:phitau2}) we get
\begin{eqnarray}
\label{e:phitau3}
\lambda(\Phi_{\tau-1}(A))&=&\sum_{k} \lambda(\Phi_{\tau-1}(A\cap W_k)) = \sum_{k} \frac{\lambda(A\cap W_k)^2}{\lambda(A)\lambda(W_k)}\nonumber\\
&=& \frac{1}{\lambda(A)}\sum_{k} \frac{\lambda(A\cap W_k)^2}{\lambda(W_k)^2}\lambda(W_k)
= \frac{1}{\lambda(A)}\E[\rho_{\tau,X}^2],
\end{eqnarray}
which completes the proof of the lemma.
\end{proof}

Now we are ready to prove Theorem \ref{t:stretchingfixed}.

\begin{proof}[Proof of Theorem \ref{t:stretchingfixed}]
Consider $\tau_1$, $\tau_2$, $\tau_3$, $\tau_4$ as in Theorem \ref{t:stoptime}. It follows from Theorem \ref{t:stoptime} that one of the following three cases must hold.

\begin{enumerate}
\item[(i)] $\P[\tau=\tau_1]\geq \frac{1}{6}$.
\item[(ii)] $\P[\tau=\tau_2]\geq \frac{1}{3}$.
\item[(iii)] $\P[\tau=\tau_3]\geq \frac{1}{6}$.
We treat each of these cases separately.
\end{enumerate}
{\bf Case 1:} $\P[\tau=\tau_1]\geq \frac{1}{6}$.

In this case it follows that
$$\E[\sum_{i=1}^{\tau-1} \tilde{\Delta}_{i,X}^2]\geq \frac{\varepsilon_1}{6}.$$

Now observe that using Lemma \ref{l:deltatildedelta} we have

\begin{eqnarray}
\label{deltatau}
\E\left[\sum_{i=1}^{\tau-1}\tilde{\Delta}_{i,X}^2\right]&=& \E\left[\sum_{i}\tilde{\Delta}_{i,X,\tau}^2 1_{\{\tau\geq i+1\}}\right]\nonumber\\
&=& \E\left(\sum_{i}\E[\tilde{\Delta}_{i,X,\tau}^2 1_{\{\tau\geq i+1\}}\mid \cf_{i+1}]\right)\nonumber\\
&\leq & C_{11}\E\left[\sum_{i}\Delta_{i,X}^2 1_{\{\tau\geq i+1\}}\right]\nonumber\\
&=& C_{11}\E\left[\sum_{i=1}^{\tau-1}\Delta_{i,X}^2\right].
\end{eqnarray}

It follows using Obervation \ref{o:martingalevariance} that
$${\rm Var}(\rho_{\tau,X})\geq \frac{\varepsilon_1}{6C_{11}}.$$

It follows now from Lemma \ref{l:phitau} that
$$\lambda(\Phi_{\tau-1}(A))\geq \lambda(A)+2\varepsilon$$
where $\varepsilon$ is a fixed constant smaller than $\frac{\varepsilon_1}{12\gamma'C_{11}}$. Choose $m$ sufficiently large so that $\P[\tau<m/2]<\varepsilon$. It then follows that

$$\lambda(\Phi_{(\tau-1)\wedge m}(A))\geq \lambda(A)+\varepsilon.$$

Also, it follows from Lemma \ref{l:phitaulip} and the definition of $\tau_4$ that $||\Phi'_{(\tau-1)\wedge m}-{\rm \mathbf{I}}||<\eta/100$. Since $\Phi_{(\tau-1)\wedge m}$ is continuously differentiable except on finitely many curves, it follows that we have $\phi:=\Phi_{(\tau-1)\wedge m}$ is bi-Lipschitz with Lipschitz constant $1+\eta$. So the proof of Theorem \ref{t:stretchingfixed} is finished in this case.

{\bf Case 2:} $\P[\tau=\tau_2]\geq \frac{1}{3}$.

In this case we have $\P[\tau_2 < \infty]\geq \frac{1}{3}$. Hence it follows that there exists $x_1,x_2,\ldots x_n\in \Lambda_{1}$ and $i_1,i_2,\ldots ,i_n>0$ such that $\Lambda_{i_{k},x_{k}}$ are disjoint (except may be at the boundary), $|\Delta_{i_{k},x_{k}}|\geq \sqrt{\varepsilon_2}$, for each $k$ and $\sum_{k=1}^{n}\lambda(\Lambda_{i_{k},x_{k}})\geq \frac{1}{12}$. Define the function $\phi$ as follows. Set $\phi=\Psi_{\sqrt{\varepsilon_2}, \Lambda_{i_k,x_k},\rightarrow}$ on $\Lambda_{i_k,x_k}$ if $i_k$ is odd and $\phi=\Psi_{\sqrt{\varepsilon_2}, \Lambda_{i_k,x_k},\uparrow}$ on $\Lambda_{i_k,x_k}$ if $i_k$ is even. Set $\phi$ to be identity on $\Lambda_1\setminus (\cup_{k} \Lambda_{i_k,x_k})$. It is clear that such a $\phi$ is well-defined, identity on the boundary of $\Lambda_1$ and is bi-Lipschitz with Lipschitz constant $(1+\eta)$ by choosing $\varepsilon_2$ sufficiently small. Now observe that
$$\lambda(\phi(A\cap \Lambda_{x_k,i_k}))-\lambda(A\cap \Lambda_{i_k,x_k}))\geq \lambda(\Lambda_{i_k,x_k})){\varepsilon_2}.$$
Summing over $k$ we get that
$$\lambda(\phi(A))\geq \lambda(A)+ \frac{{\varepsilon_2}}{12}.$$
So the conclusion of Theorem \ref{t:stretchingfixed} holds for $\varepsilon< \frac{\varepsilon_2}{12}$.

{\bf Case 3:} $\P[\tau=\tau_3]\geq \frac{1}{6}$.

In this case also it follows that ${\rm Var}(\rho_{\tau,X})\geq \frac{\varepsilon_{3}^2}{6}$. Arguing as in case 1, it follows that in this case also there is a bi-Lipschitz bijection $\phi$ with Lipschitz constant $1+\eta$ such that
$$\lambda(\phi(A))=\lambda(A)+\varepsilon$$
where $\varepsilon$ is a constant smaller than $\frac{\varepsilon_3^{2}}{12}$.

This completes the proof of Theorem \ref{t:stretchingfixed}.
\end{proof}

\bibliography{stretching}
\bibliographystyle{plain}

\appendix
\gdef\thesection{Appendix \Alph{section}}
\section{Estimates for $g_{r}$ and $\Psi_{\delta}$}
\label{s:appa}
In this appendix we provide the proofs of Lemma \ref{l:psideltacont}, Lemma \ref{l:grthetasmoothnesstheta}, Lemma \ref{l:grthetasmoothnessr}, Lemma \ref{l:2ndd1}, Lemma \ref{l:2ndd2} and Lemma \ref{l:2ndd3}.

\begin{proof}[Proof of Lemma \ref{l:psideltacont}]
Let $\tilde{\Lambda}^{1}=[0,1]\times [0,\frac{1}{2})$ and $\tilde{\Lambda}^{2}=[0,1]\times (\frac{1}{2},1]$.

{\bf Step 1}: $\Psi_{\delta}$ is continuous on $\tilde{\Lambda}^1\cup \tilde{\Lambda}^{2}$.

Notice that it is clear that $\Psi_{\delta}$ is continuous at $(1/2,0)$. Hence it suffices to prove that for $g_r$ defined by (\ref{e:grtheta1}) and (\ref{e:grtheta2}) we have $(r,\ell)\rightarrow g_r(\ell)$ is continuous. Without loss of generality assume $l\leq 0$. Define
\begin{equation}
\label{e:integralH}
H_r(\ell)=(1+h'(r))\ell-\frac{h'(r)\sin(\ell\Theta(r))}{\Theta(r)}.
\end{equation}
Notice that it follows from (\ref{e:grtheta1}) that
\begin{equation}
\label{e:grtheta3}
(1+\delta)(H_{r}(\ell)-H_{r}(-1))=H_{r}(g_r(\ell))-H_{r}(-1)
\end{equation}
and the assertion follows from the continuity of $H_{r}$.

{\bf Step 2:} Let $x=(x_1,\frac{1}{2})$ with $x\in [0,\frac{1}{2}]$. Then $\Psi_{\delta}$ is contnuous at $x$.

Without loss of generality assume $x_1\leq \frac{1}{2}$. Take $u_n=(u_1^n,u_2^n)\in [0,1]^2$ converging to $x$. Without loss of generality assume $\{u_{n}\}\subseteq \tilde{\Lambda}^{1}$. Let $(r_n,\ell_n)=K(u_n)$. Then $(r_n,\ell_n)\rightarrow (\frac{1}{2}, 2x_1-1)$. To prove that $\ell_n\rightarrow \ell= 2x_1-1$, observe that,
$$u_2^n=\frac{1}{2}+(r_n+h(r_n))\sin (\ell\Theta(r_n)).$$
Taking limit as $r_n\rightarrow \frac{1}{2}$ we get the result. Now taking limit as $r_n\rightarrow \frac{1}{2}$, and $\ell_n\rightarrow \ell$ in (\ref{e:integralH}) we get
$(1+\delta)(\ell+1)=\lim(g_r(\ell)+1)$, which proves Step 2.
\end{proof}

\begin{proof}[Proof of Lemma \ref{l:grthetasmoothnesstheta}]
Without loss of generality fix $\ell\in [-1,0)$. Observe that

\begin{equation}
\label{e:grthetasmoothness3}
\delta \int_{-1}^{\ell} (1+h'(r)-h'(r)\cos(\theta\Theta(r)))~d\theta= \int_{\ell}^{g_{r}(\ell)} (1+h'(r)-h'(r)\cos(\theta\Theta(r)))~d\theta.
\end{equation}

Using mean value theorem it follows that there exists $\theta^*\in (-1,\ell), \theta^{**}\in (\ell,g_{r}(\ell))$ such that
$$\delta(\ell+1) (1+h'(r)-h'(r)\cos(\theta^*\Theta(r)))= (g_r(\ell)-\ell) (1+h'(r)-h'(r)\cos(\theta\Theta^{**}(r))).$$

It follows that there exists a constant $C>0$, such that we have for all $r,\theta$,

\begin{equation}
\label{e:grthetasmoothness2}
(g_r(\ell)-\ell)\leq C\delta(\ell+1)
\end{equation}

since $h'(r)=O(h(r)^2)$ as $r\rightarrow \frac{1}{2}$. Moreover, since $h'(r)=o(h(r)^2)$ as $r\rightarrow \frac{1}{2}$, we have

\begin{equation}
\label{e:grthetadelta}
(g_r(\ell)-\ell)=\delta(\ell+1)(1+o(1))
\end{equation}
as $r\rightarrow \frac{1}{2}$.

Differentiating the integral equation (\ref{e:grthetasmoothness3}), we get
$$(1+\delta)(1+h'(r)-h'(r)\cos(\ell\Theta(r)))=\frac{\partial g_r}{\partial \ell}(1+h'(r)-h'(r)\cos(g_r(\ell)\Theta(r))). $$

It follows that

\begin{equation}
\label{e:grthetasmoothness4}
(\frac{\partial g_r}{\partial \ell}-1)=\delta\dfrac{(1+h'(r)-h'(r)\cos(\ell\Theta(r)))+\delta^{-1} h'(r)(\cos(g_r(\ell)\Theta(r))-\cos(\ell\Theta(r)))}{(1+h'(r)-h'(r)\cos(g_r(\ell)\Theta(r)))}.
\end{equation}

It follows now using (\ref{e:grthetasmoothness2}) that since $h'(r)=O(h(r))^2$ as $r\rightarrow \frac{1}{2}$, there is a constant $C>0$ such that $\sup_{r,\theta} |\frac{\partial g_{r}(\theta)}{\partial \theta}-1|\leq C\delta$.
\end{proof}

\begin{proof}[Proof of Lemma \ref{l:grthetasmoothnessr}]
Without loss of generality we again assume that $\theta\leq 0$. Observe that by differentiating both sides of (\ref{e:grtheta3}) w.r.t. $r$ we get that

\begin{equation}
\label{e:boundpd1}
\frac{\partial g_{r}(\theta)}{\partial r}\left[\frac{\partial H_{r}
(\ell)}{\partial \ell}\right]_{\ell=g_r(\theta)}+\left[\frac{\partial H_r(\ell)}{\partial r}\right]_{\ell=g_r(\theta)}- \frac{\partial H_r(\theta)}{\partial r}=\delta\left(\frac{\partial H_r(\theta)}{\partial r}-\frac{\partial H_r(-1)}{\partial r}\right).
\end{equation}

It follows that there exists $\theta^*\in (\theta, g_r(\theta))$ such that the left hand side of (\ref{e:boundpd1}) reduces to

$$\frac{\partial g_{r}(\theta)}{\partial r}\left(1+h'(r)-h'(r)\cos(g_r(\theta)\Theta(r))\right)+ (g_r(\theta)-\theta)(h''(r)-h''(r)\cos(\theta^*\Theta(r))+h'(r)\Theta'(r)\theta\sin(\theta^*\Theta(r)))$$

Similarly there exists $\theta^{**}\in (-1,\theta)$ such that the right hand side of (\ref{e:boundpd1}) is equal to

$$\delta(\theta+1)(h''(r)-h''(r)\cos(\theta^{**}\Theta(r))+h'(r)\Theta'(r)\theta\sin(\theta^{**}\Theta(r))). $$

Now observe that
$$\Theta'(r)=-\frac{1+h'(r)}{(r+h(r))\sqrt{4((r+h(r)))^2-1}}$$
for $r>r_0$ and $\Theta'(r)$ is bounded away from infinity if $r<r_0$.

Hence it follows from the above equations that as long as $h'(r)=o(h(r)^2)$ and $h''(r)=O(h(r)^3)$ as $r\rightarrow \frac{1}{2}$, there exists an absolute constant $C>0$ such that

$$|\frac{\partial g_{r}(\theta)}{\partial r}|\left(1+h'(r)-h'(r)\cos(g_r(\theta)\Theta(r))\right)\leq C\delta. $$

The final assertion follows using (\ref{e:grthetadelta}).
\end{proof}

\begin{proof}[Proof of Lemma \ref{l:2ndd1}]
Let $D(r,\theta)$ denote the determinant of the $J(r,\theta)$. Clearly since $D(r,\theta)$ is bounded away from $0$ and $\infty$ it suffices to prove the following two statements.

\begin{enumerate}
\item[(i)] $|\frac{\partial D}{\partial r}|\leq C$, $|\frac{\partial D}{\partial \theta}|\leq C$.
\item[(ii)] $||J_r(r,\theta)||\leq C, ||J_{\theta}(r,\theta)||\leq C$.
\end{enumerate}

The first assertion follows directly by differentiating $D$ since $(h'(r))^2=O(h(r)^3)$ and $h''(r)=O(h(r)^2)$ as $r\rightarrow \frac{1}{2}$.
The second assertion follows directly by differentiating $J(r,\theta)$ entrywise (w.r.t. $r$ and $\theta$) since $(h'(r))^2=O(h(r)^3)$ and $h''(r)=O(h(r)^2)$ as $r\rightarrow \frac{1}{2}$.
\end{proof}

\begin{proof}[Proof of Lemma \ref{l:2ndd2}]
Without loss of generality, we can assume $\theta \leq 0$. It suffices to prove that there exists an absolute constant $C$ such that

\begin{enumerate}
\item[(i)] $|\frac{\partial ^2 g_{r}}{\partial \theta ^2}|\leq C\delta$.
\item[(ii)] $|\frac{\partial ^2 g_{r}}{\partial \theta \partial r}|\leq C\delta$.
\item[(iii)] $|\frac{\partial ^2 g_{r}}{\partial r^2}|\leq C\delta$.
\end{enumerate}

Since the functions are sufficiently smooth the mixed partial derivatives will be equal.

For the proof of $(i)$ and $(ii)$, consider (\ref{e:grthetasmoothness4}). Call the numerator $A(r,\theta)$ and the denominator $B(r,\theta)$. Since $B(r,\theta)$ is bounded away from $0$ and $\infty$, it suffices to prove that for some absolute constant $C>0$, we have  $|\frac{\partial A}{\partial r}|\leq C\delta, |\frac{\partial A}{\partial \theta}|\leq C\delta, |\frac{\partial B}{\partial r}| \leq C, |\frac{\partial B}{\partial \theta}|\leq C$.

We have from (\ref{e:grthetasmoothness4}) that

$$\frac{\partial B}{\partial r}=h''(r)(1-\cos(g_r(\theta)\Theta(r)))+h'(r)\sin(g_r(\theta)\Theta(r))[\frac{\partial g_{r}}{\partial r}\Theta(r)+g_r(\theta)\Theta'(r)].$$

Since the functions above are bounded if $r$ is bounded away from $\frac{1}{2}$ and as $r\rightarrow \frac{1}{2}$ we have $h''(r)=O(h(r)^2)$, and $(h'(r))^2=O(h(r)^3)$ it follows using Lemma \ref{l:grthetasmoothnessr} that $|\frac{\partial B}{\partial r}| \leq C$ for some absolute constant $C$.

We also have
$$\frac{\partial B}{\partial \theta}=h'(r)\sin(g_r(\theta)\Theta(r))\Theta(r)\frac{\partial g_{r}}{\partial \theta}.$$

It follows that $|\frac{\partial B}{\partial \theta}| \leq C$ for some absolute constant $C$.

Next observe that
\begin{eqnarray*}
\frac{\partial A}{\partial r} &=& \delta[h''(r)(1-\cos(\theta \Theta (r)))+h'(r)\theta\Theta'(r)\sin(\theta \Theta (r))] + h''(r)(\cos(g_r(\theta)\Theta(r))-\cos(\theta\Theta(r)))\\
&+& h'(r)\Theta'(r)\biggl((\theta-g_r(\theta))\sin(g_r(\theta)\Theta(r))-\theta(\sin(\theta\Theta(r)) \sin(g_r(\theta)\Theta(r))\biggr)\\
&-& h'(r)\frac{\partial g_r}{\partial r}\Theta(r)\sin(g_r(\theta)\Theta(r)).
\end{eqnarray*}

As before, notice that everything is bounded if $r$ is bounded away from $0$ and $\frac{1}{2}$. It follows using Lemma \ref{l:grthetasmoothnesstheta}, Lemma \ref{l:grthetasmoothnessr} and $h'(r)^2=O(h(r)^3)$, $h''(r)=O(h(r)^2)$ as $r\rightarrow \frac{1}{2}$ that $|\frac{\partial A}{\partial r}|\leq C\delta$ for some absolute constant $C>0$.

Finally observe that

$$\frac{\partial A}{\partial \theta}=\delta h'(r)\Theta(r)\sin(\theta\Theta(r))+h'(r)\Theta(r)\left(\sin(\theta\Theta(r))-\sin(g_r(\theta)\Theta(r))-(\frac{\partial g_r}{\partial \theta}-1)\sin(g_r(\theta)\Theta(r))\right).$$

Arguing as before it follows from Lemma \ref{l:grthetasmoothnesstheta} and $h'(r)=O(h(r)^2)$ as $r\rightarrow \frac{1}{2}$ that $|\frac{\partial A}{\partial \theta}|\leq C\delta$ for some absolute constant $C>0$. This completes the proof of (i) and (ii) above.

For proof of (iii), consider (\ref{e:boundpd1}), let us denote

$$A_1(r,\theta)=\left[\frac{\partial H_r(\ell)}{\partial r}\right]_{\ell=g_r(\theta)}- \frac{\partial H_r(\theta)}{\partial r},$$
$$A_2(r,\theta)= \left(\frac{\partial H_r(\theta)}{\partial r}-\frac{\partial H_r(-1)}{\partial r}\right)$$
and
$$A_3(r,\theta)=(1+h'(r)-h'(r)\cos(g_r(\theta)\Theta(r))).$$

Observe that $A_3$ is bounded away from $0$ and $\infty$, and there exists a constant $C$ such that $|\frac{\partial A_3}{\partial r}|\leq C$ since $h''(r)=O(h(r)^2)$ and $h'(r)^2=O((h(r))^3)$ as $r\rightarrow \frac{1}{2}$. Hence using (\ref{e:boundpd1}), it suffices to show that for some absolute constant $C$ such that $|\frac{\partial A_1}{\partial r}|\leq C\delta$ and
$|\frac{\partial A_2}{\partial r}|\leq C$.

Observe that

$$\frac{\partial A_2}{\partial r} =\int_{-1}^{\theta} \frac{\partial ^2}{\partial r^2}(1+h'(r)-h'(r)\cos(\theta\Theta(r)))~d\theta$$

and

\begin{eqnarray*}
\frac{\partial A_1}{\partial r} &=& \int_{\theta}^{g_r(\theta)} \frac{\partial ^2}{\partial r^2}(1+h'(r)-h'(r)\cos(\theta\Theta(r)))~d\theta\\
&+& \frac{\partial g_r(\theta)}{\partial r}\left(h''(r)(1-\cos(g_r(\theta)\Theta(r)))+h'(r)\Theta'(r)g_r(\theta)\sin(g_r(\theta)\Theta(r))\right).
\end{eqnarray*}

Arguing as before, using Lemma \ref{l:grthetasmoothnessr}, $h''(r)=O(h(r)^2)$, $h'(r)^2=O(h(r))^3$ as $r\rightarrow \frac{1}{2}$ it follows that it is in fact enough to show that $\frac{\partial ^2}{\partial r^2}(1+h'(r)-h'(r)\cos(\theta\Theta(r)))$ is bounded. This follows directly by differentiating since $h^{(3)}(r)=O(h(r)^2)$ and $h''(r)h'(r)=O(h(r)^3)$ as $r\rightarrow \frac{1}{2}$ (the second derivative remains bounded if $r$ is bounded away from $0$ and $\frac{1}{2}$).

This completes the proof of the Lemma.
\end{proof}

\begin{proof}[Proof of Lemma \ref{l:2ndd3}]
Notice that any real valued smooth function $F=F(r,\theta)$, let $\tilde{F}$ denote the function $\tilde{F}(r,\theta)=F(r,g_r(\theta))$. Then we have
$$\frac{\partial \tilde{F}}{\partial r}= \left[\frac{\partial F}{\partial r}\right]_{r,g_r(\theta)}+\left[\frac{\partial F}{\partial  \theta}\right]_{r,g_r(\theta)}\frac{\partial g_r}{\partial r}.$$

$$\frac{\partial \tilde{F}}{\partial \theta}=\left[\frac{\partial F}{\partial  \theta}\right]_{r,g_r(\theta)}\frac{\partial g_r}{\partial \theta}.$$

Now the lemma follows from Lemma \ref{l:grthetasmoothnesstheta}, Lemma \ref{l:grthetasmoothnessr} and the fact that $||J_r(r,\theta)||\leq C, ||J_{\theta}(r,\theta)||\leq C$ which was established in the proof of Lemma \ref{l:2ndd1}.
\end{proof}

\end{document}